\documentclass{amsart}
\usepackage[margin=1in]{geometry}
\usepackage{amsmath}
\usepackage{amssymb}
\usepackage{amsthm}
\usepackage{mathrsfs}
\usepackage{mathtools}
\usepackage{enumerate}
\usepackage{tikz}
\usetikzlibrary{decorations.pathreplacing}
\usepackage[pagebackref]{hyperref}
\usepackage{tabularx}
\usepackage{float}
\usepackage{relsize}
\usepackage[T1]{fontenc}
\newtheorem{thm}{Theorem}[section]
\newtheorem{prop}[thm]{Proposition}
\newtheorem{lem}[thm]{Lemma}
\newtheorem{cor}[thm]{Corollary}
\theoremstyle{definition}
\newtheorem{dfn}[thm]{Definition}

\newtheorem{ex}[thm]{Example}
\newtheorem{rmk}[thm]{Remark}

\numberwithin{equation}{section}

\usepackage{xcolor}

\DeclareMathOperator{\Fav}{Fav}
\DeclareMathOperator{\Lip}{Lip}
\DeclareMathOperator{\proj}{proj}

\title[A Quantified Two-projection Theorem for Nonlinear Projections]{A Quantified Two-projection Theorem for Nonlinear Projections}

\author{Zhangze Li}
\address{Department of Mathematics, The Ohio State University, Columbus, OH}
\email{li.15128@osu.edu}

\author{Krystal Taylor}
\address{Department of Mathematics, The Ohio State University, Columbus, OH}
\email{taylor.2952@osu.edu}

\date{}

\thanks{K.T. is supported in part by the Simons Foundation Grant GR137264.
Part of this work was completed at the American Institute of Math during the SQuaRE \textit{Covering Fractals by Curves}.}



\begin{document}

\begin{abstract}
The classic Besicovitch projection theorem asserts that if a set is purely $1$-unrectifiable with finite length in $\mathbb{R}^2$, then its orthogonal projection has Lebesgue measure zero in almost every direction.
Conversely, the two-projection theorem states that if a Borel set has zero measure under orthogonal projections onto two distinct non-antipodal directions, it must be purely $1$-unrectifiable.
We extend the two-projection theorem to certain families of nonlinear projections and consider applications to pinned distance sets, radial projections, and curve projection operators.
Further, we use a multiscale framework to obtain a quantitative version of our nonlinear two-projection theorem. Our arguments utilize methods introduced by Tao, who provided a quantitative treatment of the classic linear two-projection theorem.
\end{abstract}

\maketitle

\setcounter{tocdepth}{1}
\tableofcontents

\section{Introduction and background}
Rectifiability plays a fundamental role in geometric measure theory.
Roughly speaking, rectifiable sets behave like subsets of Lipschitz graphs, while purely unrectifiable sets have no such structure.

Investigating geometric characterizations of rectifiability is a central problem in the field.
We focus on two foundational results that relate rectifiability to projection theory: the Besicovitch projection theorem and the two-projection theorem. Before stating each of these results, we give some relevant definitions.

\begin{dfn}[Lipschitz Graph]{}
    A set $\zeta\subseteq\mathbb R^2$ is called a \textit{Lipschitz graph} if there exists a Lipschitz function $F:\mathbb R \to \mathbb R$ and an orthonormal pair $\omega_1,\omega_2\in S^1$ such that $\zeta=\{x\omega_1+F(x)\omega_2:x\in\mathbb R\}$.
\end{dfn}


\begin{dfn}[Rectifiable and Unrectifiable Sets]\label{defn: rect}
    Let $E\subseteq\mathbb R^2$.
    \begin{enumerate}[\rm(1)]
        \item $E$ is called \textit{$1$-rectifiable} if there exists a countable collection of Lipschitz graphs $\{\zeta_n\}_{n\in\mathbb N}$ such that  $E$ is covered by $\{\zeta_n\}_{n\in\mathbb N}$ up to a $\mathcal{H}^1$-null set, that is, $\mathcal{H}^1\left(E\setminus \bigcup\limits_{n=1}^\infty \zeta_n\right)=0$, where $\mathcal{H}^1$ denotes the $1$-dimensional Hausdorff measure.
        \item $E$ is called \textit{purely $1$-unrectifiable} if $\mathcal{H}^1\left(E\cap \zeta\right)=0$ for every Lipschitz graph $\zeta\subseteq\mathbb R^2$. Equivalently, $E$ is purely $1$-unrectifiable if $\mathcal{L}^1(\{x\in\mathbb R: x\omega_1+F(x)\omega_2\in E\})=0$ for every Lipschitz function $F:\mathbb R \to \mathbb R$ and orthonormal pair $\omega_1,\omega_2\in S^1$.
    \end{enumerate}
\end{dfn}

Throughout the paper, $\proj_\omega$ denotes the orthogonal projection onto $\omega\in S^1$, and $\sigma$ denotes the normalized surface measure on $S^1$.

The Besicovitch projection theorem asserts the following:
\begin{thm}[Besicovitch Projection Theorem, \cite{Bes39}]\label{thm:Bes}
    Let $E\subseteq \mathbb R^2$ be a Borel set with $\mathcal{H}^1(E)<\infty$. If $E$ is purely $1$-unrectifiable, then $\mathcal{L}^1(\proj_{\omega}(E))=0$ for $\sigma$-a.e.\ $\omega\in S^1$.
\end{thm}

The Besicovitch theorem can be formulated in terms of Favard length.
The \textit{Favard length} of a Borel set $E\subseteq\mathbb R^2$ is defined as
\[\Fav(E):=\int_{S^1} \mathcal{L}^1(\proj_\omega(E))\,d\sigma(\omega).\]
So, with $E$ as in Theorem \ref{thm:Bes}, if $E$ is purely $1$-unrectifiable, then $\Fav(E)=0$.

The Besicovitch projection theorem is foundational and has inspired several major lines of further research, including (i) the Favard length problem, (ii)  the development of quantitative variants, and (iii) the study of nonlinear variants, in which the family of orthogonal projections is replaced by some other family of maps. We give a brief introduction to these rapidly developing research areas in Section \ref{sec:survey}.

In this article, we initiate further investigation into the comparatively less studied two‑projection theorem; the proof of this theorem is elementary and relies on the Lebesgue density theorem and Rademacher's differentiation theorem, see \cite{Tao09}.
\begin{thm}[Two-Projection Theorem]\label{thm:twoproj}
    Let $E\subseteq \mathbb R^2$ be a Borel set. If there exist two non-antipodal $\omega,\omega'\in S^1$ such that $\mathcal{L}^1(\proj_\omega (E))=\mathcal{L}^1(\proj_{\omega'}(E))=0$, then $E$ is purely $1$-unrectifiable.
\end{thm}

A central theme of this paper is that rectifiability can be detected from two suitably transverse nonlinear projections, and that this principle admits a quantitative multiscale analogue.

Our paper has three aims:
\begin{enumerate}
\item
We extend the classic two-projection theorem to a nonlinear setting, and show its conclusion still holds when the family of orthogonal projection maps is replaced by a family of nonlinear projections from a wide class.
See Section \ref{sec: gen two proj} and the main result, Theorem \ref{thm:two proj for general proj}.
\smallskip
\item
As an application, we verify that rectifiable $1$-sets have pinned distance sets with positive measure for all but possibly one exceptional pin; see Section \ref{app:Dist}. We obtain analogous results for curve projection operators in Section \ref{app:CP} and radial projections in Section \ref{app:RP}.
\smallskip
\item
We utilize a multiscale analysis to obtain a quantitative version of the two-projection theorem for a family of nonlinear projections in Theorem \ref{thm:quant curve two}. This result extends a theorem of Tao, and it provides a mechanism to bound the rectifiability constant. This serves as the opposite direction of the nonlinear quantitative Besicovitch theorem obtained by Davey and the second listed author in \cite{DT21}. These results are presented in Section \ref{sec:results II}, and further context is provided in Section \ref{sec:survey}.
\end{enumerate}

The results of this paper establish nonlinear analogues of classical topics in geometric measure theory and fractal geometry. Our nonlinear extensions place the two-projection theorem within the broader theory of nonlinear projections, which has recently played an important role in the study of distance sets, radial projections, curve projections, and related problems in fractal geometry. At a quantitative
level, our work builds on Tao's multiscale framework and connects nonlinear projection theory with quantitative rectifiability. In this sense, the present paper may be viewed as a bridge between qualitative projection theorems, quantitative rectifiability, and nonlinear projection theory.

To put our results into further context, we provide a summary of recent developments in the field in Section \ref{sec:survey}.
The proof of Theorem \ref{thm:two proj for general proj} (a two-projection theorem for generalized projections) is given in Section \ref{section:general two proj thm}, and an alternative proof using the area formula is provided in Section \ref{sec:contrapositive}.
The proof of Theorem \ref{thm:quant curve two} (a quantitative two-projection theorem for curve projections) is provided in Section \ref{section:proof}.

Some basic notation used throughout the paper follows.
\begin{dfn}[Asymptotic Notation]{}
    Let $X,Y\geq 0$. We write $X\lesssim Y$ if there exists a constant $C>0$ such that $X\leqslant CY$. We write $X\sim Y$ if $Y\lesssim X\lesssim Y$.
\end{dfn}

\medskip
\noindent {\bf Acknowledgment}.
 The authors wish to thank Paige Bright for many helpful comments and discussions on this article.

\section{Generalized two-projection theorems and applications}\label{sec: gen two proj}
In this section, we investigate how the two-projection theorem can be extended to more general projection families. We derive the classic two-projection theorem as a consequence, and we introduce additional conditions so that the original proof continues to hold in a broader setting.

Our main Theorem \ref{thm:two proj for general proj} essentially says that if two
sufficiently transverse projections of a set are both small, then the set
cannot contain a positive-measure rectifiable component.

\begin{dfn}[Determinant]{}
    Let $v_1,v_2\in\mathbb R^2$ be two vectors. We use $|v_1\wedge v_2|$ to denote the absolute value of the determinant of the matrix with columns $v_1,v_2$.
\end{dfn}

\begin{thm}[Two-Projection Theorem for Generalized Projections]{}\label{thm:two proj for general proj}
     Let $\Lambda\subseteq\mathbb R^2$ be a Borel set, and $P:\mathbb R^2\times \Lambda\to \mathbb R$ be a continuous function such that, for each $\lambda\in \Lambda$, the map $P_\lambda(x):=P(x,\lambda)$ is Lipschitz on every compact subset of $\mathbb R^2$.

    Let $E\subseteq \mathbb R^2$ be a compact set. Suppose there exists $\lambda,\lambda'\in \Lambda$ with $\lambda\ne\lambda'$ such that the following properties hold:
    \begin{enumerate}[\rm(1)]
        \item $\mathcal{L}^1(P_\lambda(E))=\mathcal{L}^1(P_{\lambda'}(E))=0$,
        \item the maps $P_\lambda$ and $P_{\lambda'}$ are differentiable at $\mathcal{H}^1$-a.e.\ $x\in E$. Moreover, $|\nabla P_\lambda(x)\wedge \nabla P_{\lambda'}(x)|\ne 0$ for $\mathcal{H}^1$-a.e.\ $x\in E$, that is, $\nabla P_\lambda(x)$ and $\nabla P_{\lambda'}(x)$ are not parallel for $\mathcal{H}^1$-a.e.\ $x\in E$.
    \end{enumerate}
    Then $E$ is purely $1$-unrectifiable.
\end{thm}

The proof of Theorem \ref{thm:two proj for general proj} is found in Section \ref{section:general two proj thm}.

We now consider several concrete applications of
Theorem \ref{thm:two proj for general proj} to establish results for orthogonal projections, curve projections, pinned distance sets, and radial projections.

\subsection{Application 1: orthogonal projections}\label{app:OP}
As the first application, we recover the classic two-projection theorem.

\begin{cor}[Two-Projection Theorem for Orthogonal Projections]{}
    Let $E\subseteq\mathbb R^2$ be a compact set. If there exists two non-antipodal $\omega,\omega'\in S^1$ such that $\mathcal{L}^1(\proj_\omega (E))=\mathcal{L}^1(\proj_{\omega'}(E))=0$, then $E$ is purely $1$-unrectifiable.
\end{cor}

\begin{proof}
By definition, $\nabla \proj_\omega(x)=\omega$ for every $x\in\mathbb R^2$ and $\omega\in S^1$. Furthermore, since $\omega\ne\pm\omega'$, $|\nabla \proj_\omega(x)\wedge \nabla \proj_{\omega'}(x)|=|\omega\wedge\omega'|\ne 0$ for every $x\in E$. Therefore, $E$ is purely $1$-unrectifiable.
\end{proof}

\subsection{Application 2: curve projections}\label{app:CP}
Curve projections were originally introduced by Simon and the second listed author to study Minkowski sum sets of the form $A+\Gamma$, where $A\subseteq \mathbb{R}^2$ and $\Gamma$ is a well-behaved curve \cite{ST17}.
Further applications and properties of these maps were later developed in \cite{BT20, CDT21, DT21, LMT26}.

\begin{dfn}[Curve Projections]\label{defn:curve proj}
    Let $\Gamma\subseteq \mathbb R^2$ be a curve. Given $\alpha\in\mathbb R$, define the \textit{curve projection} at $\alpha$ on $\Gamma$ as
    \begin{equation}\label{eqn:curve proj}
        \Phi_\alpha:\mathbb R^2\to\mathcal{P}(\mathbb R),\quad \Phi_\alpha(\mathbf{a}):=\left\{\beta\in\mathbb R:(\alpha,\beta)\in (\mathbf{a}+\Gamma)\cap\ell_\alpha\right\},
    \end{equation}
    where $\ell_\alpha$ denotes the vertical line $\{x=\alpha\}$ in $\mathbb R^2$.

    If $\Gamma$ can be expressed as the graph of a function over the $x$-axis, that is, $\Gamma=\left\{(t,\gamma(t))\in\mathbb R^2:t\in\mathbb R\right\}$ for some function $\gamma:\mathbb R\to\mathbb R$, we have
    \[\Phi_\alpha(\mathbf{a})=a_2+\gamma(\alpha-a_1).\]
    \begin{figure}[H]
    \centering
    \begin{tikzpicture}[scale=1.1]
            \draw[->] (-0.2,0) -- (4.5,0) node[right] {$x$};
            \draw[->] (0,-0.2) -- (0,4) node[above] {$y$};

            \draw[orange, thick] (2,0) -- (2,4) node[above, orange] {\small $\{x=\alpha\}$};

            \draw[red, thick, domain=-0.3:3, smooth, variable=\x]
            plot ({\x}, {3.5 - 0.4*(\x)^2});

            \node[red] at (0.9,3.6) {\small ${\mathbf{a}+\Gamma}$};

            \filldraw[black] (2,0) circle (1pt);
            \node[below] at (2,0) {\large $\alpha$};

            \draw[dashed] (2,1.9) -- (3,1.9);

            \filldraw[black] (2,1.9) circle (1.5pt);

            \draw[decorate,decoration={brace,amplitude=6pt}, thick] (2,0) -- (2,1.9);
            \node[left] at (1.9,0.95) {\normalsize $\Phi_\alpha(\mathbf{a})$};

            \node[fill=gray!20, rounded corners, font=\large] at (3.7,1.95)
            {$(\alpha, \Phi_\alpha(\mathbf{a}))$};
    \end{tikzpicture}
    \caption{$\Phi_\alpha(\mathbf{a})$ is the $y$-coordinate of the vertical slice of $\mathbf{a}+\Gamma$ at $\{x = \alpha\}$.}\label{curve proj}
    \end{figure}
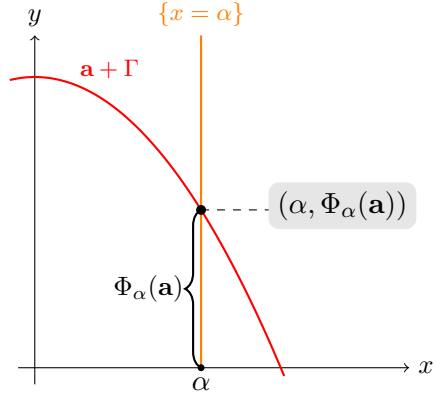
\end{dfn}

With this set-up, the following is a consequence of Theorem \ref{thm:two proj for general proj}.
\begin{cor}[Two-Projection Theorem for Curve Projections]\label{cor:curve}
    Let $E\subseteq \mathbb R^2$ be a compact set, and $\Gamma=\{(t,\gamma(t)):t\in \mathbb R\}$ be a curve such that $\gamma$ is $C^1$ and $\gamma'$ is injective. If there exists $\alpha,\alpha'\in\mathbb R$ such that $\alpha\ne\alpha'$ and $\mathcal{L}^1(\Phi_\alpha(E))=\mathcal{L}^1(\Phi_{\alpha'}(E))=0$, then $E$ is purely $1$-unrectifiable.
\end{cor}

\begin{proof}
    Since $\gamma$ is $C^1$, $\Phi_\alpha$ is Lipschitz on every compact subset of $\mathbb R^2$. Moreover,
    \[\nabla\Phi_\alpha(\mathbf{a})=(-\gamma'(\alpha-a_1),1),\quad\text{for all}\;\;\mathbf{a}\in \mathbb R^2,\alpha\in \mathbb R.\]
    Since $\gamma'$ is injective, $\gamma'(\alpha-a_1)\ne \gamma'(\alpha'-a_1)$ for every $a_1\in\mathbb R$. It follows that $|\nabla\Phi_\alpha (\mathbf{a})\wedge \nabla\Phi_{\alpha'}(\mathbf{a})|\ne 0$ for every $\mathbf{a}\in E$. Therefore, $E$ is purely $1$-unrectifiable.
\end{proof}

Corollary \ref{cor:curve} can also be obtained as a consequence of Theorem \ref{thm:quant curve two}, our quantitative two-projection theorem. We record this alternative proof below in Corollary \ref{altCor}.

Observe that the hypothesis that $\gamma'$ is injective is necessary; for instance the statement is no longer true when $\Gamma=\{(t,t^3):t\in\mathbb R\}$.
Corollary \ref{cor:curve} has implications for large unions of curves; see the follow-up paper \cite{BBMT26} for more details.

\subsection{Application 3: pinned distance sets}\label{app:Dist}
Let $E\subseteq \mathbb R^2$ be a compact set. For each $y\in \mathbb R^2$, the \textit{pinned distance set} of $E$ at $y$ is defined as
\[\Delta_y(E):=\{|x-y|:x\in E\}.\]

\begin{cor}[Two-Projection Theorem for Pinned Distance Sets with Constraint]{}
    Let $E\subseteq \mathbb R^2$ be a compact set. If there exists $y,y'\in\mathbb R^2$ with $y\ne y'$,
    and the following properties hold:
    \begin{enumerate}[\rm(1)]
        \item $\mathcal{H}^1(E\cap \ell_{y,y'})=0$, where $\ell_{y,y'}$ is the line through $y$ and $y'$,
        \item and $\mathcal{L}^1(\Delta_y(E))=\mathcal{L}^1(\Delta_{y'}(E))=0$,
    \end{enumerate}
    then $E$ is purely $1$-unrectifiable.
\end{cor}

\begin{proof}
    Let $f:\mathbb R^2\times \mathbb R^2\to\mathbb R$ be defined as
    \[f(x,y):=|x-y|^2,\quad\text{for all}\;\;x,y\in\mathbb R^2.\]
    For each $y\in \mathbb R^2$, we write $f_y(x):=f(x,y)$. Then $f_y$ is Lipschitz on every compact subset of $\mathbb{R}^2$. Moreover,
    \[\nabla f_y (x)=2(x-y),\quad\text{for all}\;\;x,y\in\mathbb R^2.\]
    
    Since $t\rightarrow t^2$ is bi-Lipschitz on every bounded interval that does not contain $0$, we 
    observe that $\mathcal{L}^1(\Delta_y(E))=0$ if and only if $\mathcal{L}^1(f_y(E))=0$. Since $\mathcal{H}^1\left(E\cap\ell_{y,y'}\right)=0$, the vectors $x-y$ and $x-y'$ are not parallel for $\mathcal{H}^1$-a.e.\ $x\in E$. It follows that $|\nabla f_y (x)\wedge\nabla f_{y'}(x)|\ne 0$ for $\mathcal{H}^1$-a.e.\ $x\in E$. Therefore, $E$ is purely $1$-unrectifiable.
\end{proof}

In fact, the conclusion still holds even without the constraint $\mathcal{H}^1\left(E\cap\ell_{y,y'}\right)=0$. This can be seen by a simple geometric lemma below.
\begin{lem}[Tangent Direction of Lipschitz Graph Along a Line]{}\label{lem:tangent vec of Lip along a line}
    Let $E\subseteq\mathbb R^2$ be a compact set, and $\zeta=\{(x,F(x)):x\in\mathbb R\}$ be a Lipschitz graph. Fix a line $\ell\subseteq\mathbb R^2$, and suppose $\mathcal{H}^1\left(E\cap\zeta\cap \ell\right)>0$. Let $A_\ell=\{x\in\mathbb R:(x,F(x))\in E\cap \ell\}$. Then $F'(x)$ exists, and the tangent vector $(1,F'(x))$ is parallel to $\ell$ for $\mathcal{L}^1$-a.e.\ $x\in A_\ell$. As a result, for $\mathcal{L}^1$-a.e.\ $x\in A_\ell$ and every $y\in \ell$ with $(x,F(x))\ne y$, we have
    \[((x,F(x))-y)\cdot (1,F'(x))\ne 0.\]
    An analogous conclusion holds for the case $\zeta=\{(F(x),x):x\in\mathbb R\}$.
\end{lem}

\begin{proof}
    Without loss of generality, we may assume $E\subseteq\ell$. We further assume $\zeta\ne\ell$, so that $\mathcal{H}^1(\ell\setminus \zeta)>0$; otherwise, the conclusion follows immediately. Since $\mathcal{H}^1\left(E\cap\zeta\right)>0$, $\mathcal{L}^1(A_\ell)>0$. Define  
    \[B_\ell:=\{x\in A_\ell:x\ \text{is a density point of $A_\ell$ and $F'(x)$ exists}\}.\]
    By the Lebesgue differentiation theorem and Rademacher's differentiation theorem, $\mathcal{L}^1(B_\ell)=\mathcal{L}^1(A_\ell)$. We show that $(1,F'(x))$ is parallel to $\ell$ for every $x\in B_\ell$.

    Fix $x_0\in B_\ell$ and $y\in\ell\setminus\zeta$. Let $v$ be a normal vector of $\ell$, and let $g:\mathbb R\to\mathbb R$ be defined as
    \[g(x):=((x,F(x))-y)\cdot v.\]
    Then $g'(x_0)$ exists. Since $\mathcal{L}^1(A_\ell)>0$ and $x_0$ is a density point of $A_\ell$, we can extract a sequence $\{x_n\}_{n\in\mathbb Z^+}\subseteq A_\ell$ such that $x_n\to x_0$ as $n\to\infty$. Since $(x,F(x))\in\ell$ for every $x\in A_\ell$, $g(x_n)=g(x_0)=0$. This suggests that
    \[g'(x_0)=\lim\limits_{n\to\infty} \frac{g(x_0)-g(x_n)}{x_0-x_n}=0.\]
    Note that $g'(x_0)=(1,F'(x_0))\cdot v$, so $(1,F'(x_0))$ is parallel to $\ell$.
\end{proof}

By Lemma \ref{lem:tangent vec of Lip along a line}, we obtain the following two-projection theorem for pinned distance sets. The proof of Corollary \ref{cor:two proj for pinned dist} is similar in nature to that of Theorem \ref{thm:two proj for general proj}, and for this reason we include both proofs in Section \ref{section:general two proj thm}.
\begin{cor}[Two-Projection Theorem for Pinned Distance Sets]{}\label{cor:two proj for pinned dist}
    Let $E\subseteq\mathbb R^2$ be a compact set. If there exists $y,y'\in \mathbb R^2$ with $y\ne y'$ and $\mathcal{L}^1(\Delta_y(E))=\mathcal{L}^1(\Delta_{y'}(E))=0$, $E$ is purely $1$-unrectifiable.
\end{cor}

We note that a higher dimensional variant of Corollary \ref{cor:two proj for pinned dist} is obtained in \cite{BBMT26} for pinned distance sets corresponding to $E\subseteq \mathbb{R}^d$ for $d\geqslant 2$. While the proof of Corollary \ref{cor:two proj for pinned dist} relies on a direct modification of the original proof of the two-projection theorem, the proof in \cite{BBMT26} depends on generalizing an inequality of Federer, which gives control of the $\mathcal{H}^m$-measure of rectifiable sets in terms of their projections onto orthogonal systems of $m$-flats. For further references on pinned distance sets, see \cite{BFOP26, IosTayUri, Liu20, Shmerkin}.

\subsection{Application 4: radial projections}\label{app:RP}
Now we consider a variant of the two-projection theorem for radial projections, and we begin with an example that shows some caution and restrictions are required.

Given $x\in\mathbb R^2$, the radial projection $\pi_x:\mathbb R^2\setminus\{x\}\to S^1$ is defined as
\[\pi_x(y):=\frac{y-x}{|y-x|}.\]

\begin{ex}[Cautionary Example for Radial Projections]{}
    The two-projection theorem can fail for radial projections. To see this, let $x,y \in \mathbb{R}^2$ with $x\ne y$, and let $\ell_{x,y}$ be the line through them. Suppose $E$ is the line segment from $x$ to $y$ given by
    \[E=\{tx+(1-t)y:t\in [0,1]\}\subseteq \ell_{x,y}.\]
    Certainly $E$ is not purely $1$-unrectifiable but rather is $1$-rectifiable. Choose $x',y'\in \ell_{x,y}\setminus E$ with $|x-x'|=|y-y'|=1$. For every $a\in E$, we have $\pi_{x'}(a)=\frac{y-x}{|y-x|}$ and $\pi_{y'}(a)=\frac{x-y}{|x-y|}$, so that $\sigma(\pi_{x'}(E))=\sigma(\pi_{y'}(E))=0$. This serves as an example of a rectifiable set with two zero measure projections.
\end{ex}

Orponen and Sahlsten \cite{OS11} showed that the two-projection theorem holds for radial projections under an additional geometric constraint. In fact, they proved that the above counterexample is essentially the only setting in which the two-projection theorem could fail.
\begin{thm}[Two-Projection Theorem for Radial Projections, Theorem~2.1 in \cite{OS11}]{}
    Let $E\subseteq\mathbb{R}^2$ be a compact set. If there exists $x,y\in\mathbb R^2$ with $x\ne y$ such that the following properties hold:
    \begin{enumerate}[\rm(1)]
        \item $\mathcal{H}^1(E\cap \ell_{x,y})=0$, where $\ell_{x,y}$ is the line through $x$ and $y$, and
        \item $\sigma(\pi_x(E))=\sigma(\pi_y(E))=0$,
    \end{enumerate}
    then $E$ is purely $1$-unrectifiable.
\end{thm}
Their proof is similar in spirit to the proof of Theorem \ref{thm:two proj for general proj}, and it relies on the local structure of density points. In fact, the proof of Theorem \ref{thm:two proj for general proj} can be adapted to prove the above theorem. The key observation is that for $a,b\in \mathbb R^2$, we have
\[|\pi_x(a)-\pi_{x}(b)|\sim \theta_x(a,b),\]
where $\theta_x(a,b)\in[0,\pi]$ is the angle between the vectors $a-x$ and $b-x$.

\section{Quantitative two-projection theorem for curve projections }\label{sec:results II}
We now turn to our aim of quantifying the two-projection theorem for generalized projections.  Before we state the main result of this section, Theorem \ref{thm:quant curve two}, we require a quantified notion of rectifiability.
\subsection{Rectifiability constant}
The \textit{rectifiability constant} was introduced by Tao \cite{Tao09} in his study of asymptotic upper bounds for the Favard length of arbitrary unrectifiable sets.
The rectifiability constant provides a framework that measures how closely a planar set can be approximated by Lipschitz graphs. 
In this paper, we adopt this notion of the rectifiability constant - with some minor modifications - as a basic tool.

We first give a simplified characterization of unrectifiability. By the implicit function theorem, Whitney's extension theorem, and Lusin's theorem, every Lipschitz graph $\zeta$ can be locally expressed as a countable union of horizontal or vertical Lipschitz graphs, up to an $\mathcal{H}^1$-null set. Each of these graphs can be parametrized either with respect to the first coordinate, $\{(x,F(x)):x\in I\}$, or with respect to the second coordinate, $\{(F(x),x):x\in I\}$, where $I$ is a small interval. As a result, it suffices to test the unrectifiability with respect to Lipschitz graphs parametrized with the standard two coordinate axes.

\begin{lem}[Equivalent Definition of Unrectifiability]{}\label{lem:equiv def of unrect}
    Given $E\subseteq\mathbb R^2$, $E$ is purely $1$-unrectifiable if and only if $\mathcal{H}^1\left(E\cap \zeta\right)=0$ for every Lipschitz graph $\zeta\subseteq\mathbb R^2$ parametrized with the standard two coordinate axes. That is
    \[\mathcal{L}^1\left(\left\{x\in\mathbb R:(x,F(x))\in E\right\}\right)=\mathcal{L}^1\left(\left\{x\in\mathbb R:(F(x),x)\in E\right\}\right)=0\]
    for every Lipschitz function $F:\mathbb R \to \mathbb R$.
\end{lem}

\begin{dfn}[Lipschitz Constant]{}
    Let $F:\mathbb R\to\mathbb R$ be a Lipschitz function. The \textit{Lipschitz constant} $\Lip(F)$ is defined as
    \[\Lip(F)=\sup_{x\ne y} \frac{|F(x)-F(y)|}{|x-y|}.\]
\end{dfn}

\begin{dfn}[Rectifiability Constant]\label{defn: rect constant}
    Let $E\subseteq\mathbb R^2$, $\varepsilon>0$, $I\subseteq \mathbb R$ be an interval, $(\omega_1,\omega_2)\in S^1\times S^1$ be an orthonormal pair, and $F:\mathbb R\to\mathbb R$ be a Lipschitz function. Set the quantity
    \[
    R_E(\varepsilon,I,F,(\omega_1,\omega_2))\hspace{-.5mm}:=\hspace{-.5mm}\frac{\mathcal{L}^1\hspace{-1mm}\left(\left\{x\in I:\exists\,y\in[-\varepsilon,\varepsilon]\ \text{s.t.}\ x\omega_1\hspace{-1mm}+\hspace{-1mm}(F(x)+y)\omega_2\in E\right\}\right)}{\mathcal{L}^1(I)}.
    \]

    Let $r,M\geqslant 0$.
    \begin{enumerate}[\rm(1)]
        \item The \textit{rectifiability constant} $\mathcal{R}_E(\varepsilon,r,M)$ is defined as
        \[\mathcal{R}_E(\varepsilon,r,M):=\sup_{\substack{\mathcal {L}^1(I)\geqslant r,\Lip(F)\leqslant M,\\ \omega_1\perp\omega_2}} R_E(\varepsilon,I,F,(\omega_1,\omega_2)).\]
        \item Let $e_1=(1,0),e_2=(0,1)\in S^1$. The \textit{restricted rectifiability constant} $\mathcal{R}_{E,\text{res}}(\varepsilon,r,M)$ is defined as
        \[\mathcal{R}_{E,\text{res}}(\varepsilon,r,M):=\sup_{\substack{\mathcal {L}^1(I)\geqslant r,\\\Lip(F)\leqslant M}} \max\{R_E(\varepsilon,I,F,(e_1,e_2)),R_E(\varepsilon,I,F,(e_2,e_1))\}.\]
    \end{enumerate}
\end{dfn}

\begin{ex}
    Observe that $0\leqslant \mathcal{R}_{E,\text{res}}(\varepsilon,r,M)\leqslant \mathcal{R}_E(\varepsilon,r,M)\leqslant 1$. Moreover, the second inequality can be strict. For example, if $E=\{(x,x):x\in [-1,1]\}$,
    \[\frac{1}{2}=\mathcal{R}_{E,\text{res}}(\varepsilon,2,1/2)<\mathcal{R}_E(\varepsilon,2,1/2)=1,\quad\text{for all}\;\;0<\varepsilon<\frac{1}{4}.\]

    In \cite[Definition~1.10]{Tao09}, the rectifiability constant is defined as $\mathcal{R}_E(\varepsilon,r,M)$. This choice is motivated by the quantitative Besicovitch projection theorem in \cite[Theorem~1.13]{Tao09}. In particular, \cite[Lemma~5.5]{Tao09} relies on the fact that a given Lipschitz graph can be parametrized in different coordinate directions. In the quantitative two-projection setting, we rely on Lemma \ref{lem:equiv def of unrect} and do not require rotations of Lipschitz graphs for the parameterizations.
\end{ex}

The rectifiability constant quantifies the extent to which a set exhibits rectifiable or unrectifiable behavior, as established in \cite[Proposition~1.11]{Tao09}.
\begin{prop}[Equivalence of Qualitative and Quantitative Unrectifiability, Proposition~1.11 in \cite{Tao09}]{}\label{prop:qual and quant rect equiv}
    Let $E\subseteq \mathbb R^2$ be a compact set. Then $E$ is purely $1$-unrectifiable if and only if $\lim\limits_{\varepsilon\to 0} \mathcal{R}_E(\varepsilon,r,M)=0$ for every $r,M>0$.
\end{prop}

\begin{prop}[Equivalence of Qualitative and Quantitative Unrectifiability]{}\label{prop:qual and quant rect equiv XY}
    Let $E\subseteq \mathbb R^2$ be a compact set. Then $E$ is purely $1$-unrectifiable if and only if $\lim\limits_{\varepsilon\to 0} \mathcal{R}_{E,\text{res}}(\varepsilon,r,M)=0$ for every $r,M>0$.
\end{prop}

\begin{proof}
    First suppose $E$ is purely $1$-unrectifiable. Since $\mathcal{R}_{E,\text{res}}(\varepsilon,r,M)\leqslant \mathcal{R}_E(\varepsilon,r,M)$, the conclusion immediately follows from Proposition \ref{prop:qual and quant rect equiv}.

    Next suppose $\lim\limits_{\varepsilon\to 0} \mathcal{R}_{E,\text{res}}(\varepsilon,r,M)=0$ for every $r,M>0$. Assume $E$ is not purely $1$-unrectifiable. By Lemma \ref{lem:equiv def of unrect}, there exists a Lipschitz function $F:\mathbb R\to\mathbb R$ with $\Lip(F)=M$ such that either $\mathcal{L}^1(\{x\in\mathbb R:(x,F(x))\in E\})>0$ or $\mathcal{L}^1(\{x\in\mathbb R:(F(x),x)\in E\})>0$. We may assume $\mathcal{L}^1(\{x\in\mathbb R:(x,F(x))\in E\})>0$. Since $E$ is compact, there exists $R>0$ such that $E\subseteq B(0,R)$. Thus, $\mathcal{L}^1 (\{x\in[-R,R]:(x,F(x))\in E\})>0$, so
    \[\mathcal{R}_{E,\text{res}}(\varepsilon,2R,M)\geqslant (2R)^{-1}\mathcal{L}^1 (\{x\in[-R,R]:(x,F(x))\in E\})>0,\quad\text{for all}\;\;\varepsilon>0,\]
    which leads to a contradiction.
\end{proof}

\subsection{Quantitative two-projection theorem}\label{section:main quant result}
We first make some basic required assumptions on the curve $\Gamma$.

\begin{dfn}[Simple Curvature Condition]\label{dfn:simple curvature condition}
    Let $\Gamma\subseteq \mathbb R^2$ be a curve. We say that $\Gamma$ satisfies the \textit{simple curvature condition} if $\Gamma=\{(t,\gamma(t)):t\in I\}$, where $I$ is a compact interval, $\gamma:\mathbb R\to\mathbb R$ is $C^1$, $\sup\limits_{t\in \mathbb R} |\gamma'(t)|\leqslant 1$, and $\gamma'$ is $M$-Lipschitz on $I$, that is,
    \[M^{-1}|s-t|\leqslant|\gamma'(s)-\gamma'(t)|\leqslant M|s-t|,\quad\text{for all}\ s,t\in I.\]
\end{dfn}

Let $\Gamma=\{(t,\gamma(t)):t\in I\}$ satisfy this condition. Then $\gamma'$ is injective, hence monotone, and $\gamma''$ exists almost everywhere. Without loss of generality, we assume that $\gamma'$ is strictly decreasing on $I$, so that $\Gamma$ is concave down.

With this setup, we study the one-parameter family of projections $\{\Phi_\alpha\}_{\alpha\in\Lambda}$ on the unit cube $[0,1]^2$.
\begin{dfn}[Curve Projections on the Unit Cube]{}\label{dfn:curve proj on cube}
    Let $\Gamma=\{(t,\gamma(t)):t\in I\}$ be a curve satisfying the simple curvature condition. Suppose $\gamma'$ is strictly decreasing and $\Gamma$ is concave down. Let $\Lambda\subseteq\mathbb R$ be a compact set with $\Lambda-[0,1]\subseteq I$.
    \textit{The one-parameter family of curve projections} $\{\Phi_\alpha\}_{\alpha\in\Lambda}$ is defined by
    \[\Phi_\alpha:[0,1]^2\to \ell_\alpha,\quad\Phi_\alpha(\mathbf{a}):=a_2+\gamma(\alpha-a_1).\]
\end{dfn}

We now state our second main result, which can be compared to \cite[Theorem~1.17]{Tao09}.  Recall that the restricted rectifiability constant and accompanying notations are defined in Definition \ref{defn: rect constant}.

\begin{thm}[Quantitative Two-Projection Theorem for Curve Projections]\label{thm:quant curve two}
    Let $E\subseteq [0,1]^2$ be a compact set, and $\Gamma=\{(t,\gamma(t)):t\in \mathbb R\}$ be a curve satisfying the simple curvature condition. We consider the curve projections $\{\Phi_\alpha\}_{\alpha\in\Lambda}$ on $[0,1]^2$ as defined in Definition \ref{dfn:curve proj on cube}.

    Let $\alpha,\alpha'\in \Lambda$ with $\alpha\ne\alpha'$. 
    Suppose that there exists a sequence of scales
    \[0<r_N<r_{N-1}<\cdots<r_1\leqslant 1\]
    with the following properties:
    \begin{enumerate}[\rm(1)]
        \item {\rm(Scale Separation)} Each $r_n$ is a dyadic number, that is, $r_n=2^{-j_n}$ for some $j_n\in\mathbb N$. In particular,
        \[r_{n+1}\leqslant \frac{1}{2}r_n,\quad\text{for all}\;\;1\leqslant n<N.\]
        \item {\rm(Small Projections)} Let $\mathcal{N}_r(G)=\{x\in\mathbb R^2:d(x,G)<r\}$ be the open $r$-neighborhood of a set $G\subseteq\mathbb R$. Then
        \[\mathcal{L}^1(\mathcal{N}_{r_{n+1}}(\Phi_\alpha(E)))\leqslant r_n\ \text{and}\ \mathcal{L}^1(\mathcal{N}_{r_{n+1}}(\Phi_{\alpha'}(E)))\leqslant r_n,\quad\text{for all}\;\;1\leqslant n<N.\]
    \end{enumerate}
    Then
    \[\mathcal{R}_{E,\text{res}}\left(r_N,1,N^{1/100}\right)\lesssim N^{-1/100}.\]
\end{thm}

\begin{rmk}
    It is not hard to adapt the proof of Theorem \ref{thm:quant curve two} to derive a quantitative version of the two-projection theorem for a more general family of projections $\{P_\lambda\}_{\lambda\in \Lambda}$ under some additional regularity conditions. For example, one may assume for each $\lambda\in \Lambda$, $P_\lambda$ is $C^2$ with nonvanishing curvature. The main idea is to quantify the growth and decay of functions $x\mapsto P_\lambda(x,F(x))$ and $x\mapsto P_\lambda(F(x),x)$, where $F:\mathbb R\to\mathbb R$ is Lipschitz. We omit the details since the proof is largely the same.
\end{rmk}

\begin{rmk}
    If the family $\{P_\lambda\}_{\lambda\in \Lambda}$ is \textit{``rotational invariant''}, in the sense that $\Lambda\subseteq\mathbb R^2$ is rotational invariant, and
    \[\mathcal{L}^1(P_{\lambda}(E))=\mathcal{L}^1\left(P_{O(\lambda)}(O(E))\right),\quad\text{for all}\ \lambda\in \Lambda,O\in \operatorname{SO}(2),\]
    then one can prove a stronger estimate
    \[\mathcal{R}_E(r_N,1,N^{1/100})\lesssim N^{-1/100}.\]
    This applies, for example, to orthogonal projections, pinned distance sets, and radial projections.
\end{rmk}

It is straightforward to deduce our qualitative two-projection theorem - Corollary \ref{cor:curve} - from Theorem \ref{thm:quant curve two}.  While the following statement is identical to that in Corollary \ref{cor:curve}, we include it here with an alternative proof.

\begin{cor}[Two-Projection Theorem for Curve Projections]\label{altCor}
    Let $E\subseteq [0,1]^2$ be a compact set, $\Gamma=\{(t,\gamma(t)):t\in \mathbb R\}$ be a curve satisfying the simple curvature condition, and $\Lambda\subseteq\mathbb R$ be a compact set. If there 
 exist $\alpha,\alpha'\in \Lambda$ such that $\alpha\ne\alpha'$ and $\mathcal{L}^1(\Phi_\alpha(E))=\mathcal{L}^1(\Phi_{\alpha'}(E))=0$, then $E$ is purely $1$-unrectifiable.
\end{cor}

\begin{proof}
    Fix $M>0$. Note that
    \[\lim\limits_{r\to 0} \mathcal{L}^1(\mathcal{N}_r(\Phi_\alpha(E)))=\lim\limits_{r\to 0} \mathcal{L}^1(\mathcal{N}_r(\Phi_{\alpha'}(E)))=0.\]
    Set $r_1=1/2$. For every $n\geqslant  2$, recursively define
    \[r_n:=\min\left\{\frac{1}{2}r_{n-1},\max\left\{2^{-j}:j\in\mathbb N,\mathcal{L}^1(\mathcal{N}_{2^{-j}}(\Phi_\alpha(E)))\leqslant r_{n-1},\mathcal{L}^1(\mathcal{N}_{2^{-j}}(\Phi_{\alpha'}(E)))\leqslant r_{n-1}\right\}\right\}.\]
    By Theorem \ref{thm:quant curve two}
    \[\mathcal{R}_{E,\text{res}}(r_N,1,M)\leqslant \mathcal{R}_{E,\text{res}}\left(r_N,1,N^{1/100}\right)\lesssim N^{-1/100}\]
    for sufficiently large $N$. Taking $N\to\infty$ gives
    \[\lim\limits_{\varepsilon\to 0} \mathcal{R}_{E,\text{res}}(\varepsilon,1,M)=\lim\limits_{N\to \infty} \mathcal{R}_{E,\text{res}}(r_N,1,M)=0.\]
    By Proposition \ref{prop:qual and quant rect equiv XY} and appropriate rescaling, $E$ is purely $1$-unrectifiable.
\end{proof}

Combined with Theorem \ref{thm:quali besi proj thm curve} later, we obtain the following characterization of purely $1$-unrectifiable sets via curve projection.
\begin{thm}[Relation between Unrectifiability and Curve Projections]
    Let $E\subseteq [0,1]^2$ be a compact set with $\mathcal{H}^1(E)<\infty$, and $\Gamma\subseteq\mathbb R^2$ be a curve satisfying the simple curvature condition. Then $E$ is purely $1$-unrectifiable if and only if $\Fav_\Gamma(E)=0$.
\end{thm}

\section{Short survey: curve projections and Favard curve length}\label{sec:survey}

One motivation for obtaining a quantitative two-projection theorem is that it provides a way to verify the hypotheses appearing in the quantitative Besicovitch projection theorem. Roughly speaking, the former translates the smallness of a few projections into multiscale geometric information about the set, while the latter converts this geometric information into decay estimates for Favard-type quantities. Thus, the two results naturally complement one
another.

%
The purpose of this section is to put the previous section into greater context by giving an introduction to the Favard problem and its nonlinear analogues. For a more in-depth survey on these topics by the second author, see \cite{Taylor24survey}.

A primary application of the classic two-projection theorem is to verify that a given set $E\subset \mathbb{R}^2$ is purely $1$-unrectifiable, which in turn allows one to apply the Besicovitch projection theorem and conclude that $\Fav(E)=0$.

Likewise, a primary application of Tao's quantitative two-projection theorem is to verify the hypotheses of his corresponding quantitative Besicovitch theorem, which, in turn, yields quantitative bounds on Favard length (see \cite{Tao09}).

It follows from the Besicovitch projection theorem, Theorem \ref{thm:Bes}, that, if $E\subseteq\mathbb R^2$ is a purely $1$-unrectifiable compact set with $0<\mathcal{H}^1(E)<\infty$, then
\[\Fav(E)=\lim_{r\to 0}\Fav(N_r(E))=0,\]
where $N_r(E)=\{x\in\mathbb R^2:d(x,E)\leqslant r\}$ is the closed $r$-neighborhood of $E$.

It is natural to ask how fast $\Fav(N_r(E))$ decays as $r\to 0$. In recent years, there have been substantial efforts on establishing asymptotic bounds for the decay of Favard length. This
is known as the \textit{Favard length problem}.

Several works \cite{BaV10,BLV14, Marshall25, NPV10} focused on special self-similar sets, including the classic \textit{four-corner Cantor set}.
Beyond self-similar sets, estimates for more general sets were provided in \cite{Mat90,Tao09}. In particular, Tao \cite{Tao09} analyzed the proof of the Besicovitch projection theorem and quantified its main steps via multilinear analysis, to obtain upper bounds on the Favard length. Tao also gave a quantitative version of two-projection theorem to check the hypotheses of the quantitative Besicovitch projection theorem.

It is worth mentioning that the Besicovitch projection theorem can be stated as follows. If $E\subseteq\mathbb{R}^2$ is a Borel set with $0<\mathcal{H}^1(E)<\infty$ and $\Fav(E)>0$, there exists a Lipschitz graph $\zeta\subseteq\mathbb{R}^2$ such that $\mathcal{H}^1\left(E\cap\zeta\right)>0$. In recent years, several works \cite{CDOV24,Dabrowski24,Dabrowski25,MO18,Orponen21} have aimed to quantify this statement. These results study how large $\Fav(E)$ must be to guarantee that $E$ contains a big piece of Lipschitz graph, that is, there exists a Lipschitz graph $\zeta\subseteq\mathbb R^2$ such that $\mathcal{H}^1\left(E\cap \zeta\right)\gtrsim \mathcal{H}^1(E)$. This question is closely related to analytic capacity, in particular to \textit{Vitushkin's Conjecture}. We refer to \cite{CT20} for an introduction to this topic; see also \cite{Dabrowski24} for recent progress.

\subsection{Favard curve length}
It is natural to formulate the Favard problem in a nonlinear setting.  For instance, we can introduce a nonlinear variant of Favard length.

\begin{dfn}[Favard Curve Length]
    Let $E\subseteq\mathbb R^2$ and $\Gamma\subseteq\mathbb R^2$ be a curve with a family of curve projections $\{\Phi_\alpha\}_{\alpha}$ defined as in (\ref{eqn:curve proj}). The \textit{Favard curve length} of $E$ is defined as
    \[\Fav_\Gamma(E):=\mathcal{L}^2\left(\left\{(\alpha,\beta)\in\mathbb R^2:\Phi_\alpha^{-1}(\beta)\cap E\ne\varnothing\right\}\right)=\int_{\mathbb R} \mathcal{L}^1(\Phi_\alpha(E))\,d\mathcal{L}^1(\alpha).\]
\end{dfn}

We can also state a nonlinear variant of the Besicovitch projection theorem.

\begin{thm}[Besicovitch Projection Theorem for Curve Projections, \cite{ST17}]\label{thm:quali besi proj thm curve}
    Let $E\subseteq [0,1]^2$ be a compact set with $\mathcal{H}^1(E)<\infty$, and let $\Gamma\subseteq\mathbb R^2$ be a curve satisfying the simple curvature condition. If $E$ is purely $1$-unrectifiable, then $\Fav_\Gamma(E)=0$.
\end{thm}

See \cite{ST17} for detailed proof.
 For further developments on the nonlinear Favard length problem, we refer to \cite{BT20, CDT21, DT21, HovJ2Led, LMT26, ST17,Taylor24survey}.

\subsection{A quantitative Besicovitch projection theorem for curve projections}
Following the approach introduced by Tao in \cite{Tao09}, Davey and the second listed author used multiscale analysis to quantify the Besicovitch projection theorem for curve projections in \cite{DT21}. We begin with the definition of Hausdorff content.
\begin{dfn}[Hausdorff Content]
    Let $E\subseteq \mathbb R^2$ and $0\leqslant r^-<r^+<\infty$. The \textit{$1$-dimensional restricted Hausdorff content} of $E$ is defined as
\[
    \mathcal{H}^1_{r^-,r^+}(E):=\inf\left\{\sum\limits_{j=1}^\infty \text{diam}(B_j): B_j\ \text{open balls},E\subseteq\bigcup_{j=1}^\infty B_j ,r^- < r(B_j) < r^+\right\}.
\]
\end{dfn}

We now state the quantitative nonlinear Besicovitch projection theorem in \cite{DT21}. Roughly speaking, they showed that if $E$ is close to being purely $1$-unrectifiable, $\Fav_\Gamma(E)$ is very small.
\begin{thm}[Quantitative Besicovitch Projection Theorem for Curve Projections, \cite{DT21}]\label{thm:quant curve besi}
    Let $E\subseteq [0,1]^2$ be a compact set with $\mathcal{H}^1(E)\leqslant L$ for some $L>0$. Suppose that there exists a sequence of scales
    \[0<r_N^-<r_N^+<r_{N-1}^-<r_{N-1}^+<\cdots<r_1^-<r_1^+\leqslant 1\]
    with the following properties:
    \begin{enumerate}[\rm(1)]
        \item {\rm(Uniform Length Bound)} $\mathcal{H}_{r_n^-,r_n^+}^1(E)\leqslant L$ for every $1\leqslant n\leqslant N$.
        \item {\rm(Scale Separation)} $r_{n+1}^+\leqslant \frac{1}{2}r_n^-$ for every $1\leqslant n< N$.
        \item {\rm(Near Unrectifiability)} $\mathcal{R}_E\left(r_{n+1}^+,r_n^-,\frac{1}{r_n^-}\right)\leqslant N^{-1/100}$.
    \end{enumerate}
    Let $\Gamma\subseteq\mathbb R^2$ be a curve satisfying the simple curvature condition. Then
    \[\Fav_\Gamma(E)\lesssim N^{-1/100}L.\]
\end{thm}
The qualitative version of Besicovitch projection theorem for curve projections follows as a direct application of Theorem \ref{thm:quant curve besi}; see \cite[Theorem~2.1]{DT21} for a detailed proof.

In Section \ref{sec:results II}, we gave a quantitative two-projection theorem for curve projections, which goes in the opposite direction of Theorem \ref{thm:quant curve besi}. More precisely, while Theorem \ref{thm:quant curve besi} shows that quantitative unrectifiability forces curve projections to be small, our result shows that if two curve projections are small, the set must have small intersection with every Lipschitz graph.

\section{Proof of Theorem \ref{thm:two proj for general proj} and Corollary \ref{cor:two proj for pinned dist}}\label{section:general two proj thm}
We now prove Theorem \ref{thm:two proj for general proj}. The proof relies on the Lebesgue differentiation theorem, Rademacher's differentiation theorem, and the intermediate value theorem to achieve a contradiction. Moreover, this proof will serve as a model to prove Theorem \ref{thm:quant curve two}. Our proof is very similar to that in \cite{Tao09}, but we provide additional details and treat a more general setting.

\begin{proof}[Proof of Theorem \ref{thm:two proj for general proj}]{}
    Assume $E$ is not purely $1$-unrectifiable. Then by Lemma \ref{lem:equiv def of unrect}, there exists a Lipschitz graph $\zeta$ such that $\mathcal{H}^1\left(E\cap \zeta\right)>0$, with either $\zeta=\{(x,F(x)):x\in\mathbb R\}$ or $\zeta=\{(F(x),x):x\in\mathbb R\}$ for some Lipschitz function $F:\mathbb R\to\mathbb R$.

    We first suppose $\zeta=\{(x,F(x)):x\in\mathbb R\}$. Let $A=\left\{x\in\mathbb R:(x,F(x))\in E\right\}$. Then $A$ is a compact set, and it follows that
    $\mathcal{L}^1(A)>0$. By the Lebesgue differentiation theorem and Rademacher's differentiation theorem, there exists $x_0\in A$ such that $F'(x_0)$ exists and
    \[\lim\limits_{r\to 0}\frac{\mathcal{L}^1\left(A\cap [x_0-r,x_0+r]\right)}{2r}=1.\]
    By condition (2), we can further assume
    \[|\nabla P_\lambda(x_0,F(x_0))\wedge \nabla P_{\lambda'}(x_0,F(x_0))|\ne 0.\]
    It follows that
    \[\text{either}\quad\nabla P_\lambda(x_0,F(x_0))\cdot (1,F'(x_0)) \ne 0\quad\text{or}\quad\nabla P_{\lambda'}(x_0,F(x_0))\cdot (1,F'(x_0))\ne 0.\]
    Indeed, if both quantities are equal to $0$, $(1,F'(x_0))$ is orthogonal to both gradient vectors $\nabla P_\lambda(x_0,F(x_0))$ and $\nabla P_{\lambda'}(x_0,F(x_0))$. Hence, $\nabla P_\lambda(x_0,F(x_0))$ and $\nabla P_{\lambda'}(x_0,F(x_0))$ are parallel, which leads to a contradiction. Without loss of generality, we may assume $\nabla P_\lambda(x_0,F(x_0))\cdot(1,F'(x_0))\ne 0$. We show that $\mathcal{L}^1\left(\{P_\lambda(x,F(x)): x\in A\}\right)>0$, from which it will follow that $\mathcal{L}^1(P_\lambda(E))>0$.

    Since $E\subseteq\mathbb R^2$ and $A\subseteq \mathbb R$ is compact, there exists $R>1$ such that $E\subseteq [-R,R]^2$ and $A\subseteq (-R+1,R-1)$, so $x_0\in (-R+1,R-1)$. Let $\varphi:[-R,R]\to \mathbb R$ be defined as
    \[\varphi(x):=P_\lambda(x,F(x)).\]
    Since $P_\lambda$ is Lipschitz on $[-R,R]^2$, $\varphi$ is Lipschitz. Also, $\varphi$ is differentiable at $x_0$, and $\varphi'(x_0)=\nabla P_\lambda(x_0,F(x_0))\cdot(1,F'(x_0))\ne 0$. We may assume $\varphi'(x_0)>0$ by replacing $\varphi$ by $-\varphi$ if necessary.

    Since $\varphi'(x_0)$ exists,
    \[\varphi'(x_0)=\lim_{x\to x_0}\frac{\varphi(x)-\varphi(x_0)}{x-x_0}.\]
    Thus, there exists $0<r_0<1$ such that if $0<|x-x_0|< r_0$,
    \[\left|\frac{\varphi(x)-\varphi(x_0)}{x-x_0}-\varphi'(x_0)\right|<\frac{\varphi'(x_0)}{4}.\]
    This shows that for every $0<r<r_0$,
    \[\varphi(x_0+r)>\varphi(x_0)+\frac{3}{4}\varphi'(x_0)r\quad \text{and}\quad\varphi(x_0-r)<\varphi(x_0)-\frac{3}{4}\varphi'(x_0)r.\]
    By the intermediate value theorem, $\varphi([x_0-r,x_0+r])$ is an interval. Since $\varphi(x_0+r)>\varphi(x_0-r)$,
    \[\mathcal{L}^1\left(\varphi([x_0-r,x_0+r])\right)\geqslant \varphi(x_0+r)-\varphi(x_0-r)>\frac{3}{2}\varphi'(x_0)r.\]

    Since $x_0$ is a density point of $A$, for $0<\varepsilon<\min\left\{1,\frac{3\varphi'(x_0)}{8\Lip(\varphi)}\right\}$, there exists $0<r_1<r_0$ such that
    \[\mathcal{L}^1\left(A\cap [x_0-r_1,x_0+r_1]\right)\geqslant 2r_1(1-\varepsilon),\]
    so
    \[\mathcal{L}^1\left([x_0-r_1,x_0+r_1]\setminus A\right)\leqslant 2r_1\varepsilon.\]
    Since $\varphi$ is Lipschitz,
    \[\mathcal{L}^1\left(\varphi([x_0-r_1,x_0+r_1]\setminus A)\right)\leqslant\Lip(\varphi)\mathcal{L}^1\left([x_0-r_1,x_0+r_1]\setminus A\right)\leqslant 2\Lip(\varphi)r_1\varepsilon,\]
    so
    \begin{align*}
        \mathcal{L}^1(\varphi(A&\cap [x_0-r_1,x_0+r_1]))\\&\geqslant \mathcal{L}^1\left(\varphi([x_0-r_1,x_0+r_1])\right)-\mathcal{L}^1\left(\varphi([x_0-r_1,x_0+r_1]\setminus A)\right)\\
        &> \frac{3}{2}\varphi'(x_0)r_1-2\Lip(\varphi) r_1\varepsilon\\
        &>\frac{3}{4}\varphi'(x_0)r_1.
    \end{align*}
    Finally, observe that
    \[\mathcal{L}^1(P_\lambda(E))\geqslant \mathcal{L}^1\left(\varphi\left(A\cap [x_0-r_1,x_0+r_1]\right)\right)>\frac{3}{4}\varphi'(x_0)r_1>0.\]
    This leads to a contradiction.

    Next, suppose $\zeta=\{(F(x),x):x\in\mathbb R\}$. The proof is essentially the same as above. Let $A=\{x\in \mathbb R:(F(x),x)\in E\}$. We choose desired $x_0\in A$ as before. The required condition still holds: either $\nabla P_\lambda(F(x_0),x_0)\cdot (F'(x_0),1)\ne 0$ or $\nabla P_{\lambda'}(F(x_0),x_0)\cdot (F'(x_0),1)\ne 0$. Without loss of generality, we may assume $\nabla P_\lambda(F(x_0),x_0)\cdot(F'(x_0),1)\ne 0$, and consider $\varphi:\mathbb R\to\mathbb R$ defined as $\varphi(x)=P_\lambda(F(x),x)$ to reach a contradiction.
\end{proof}

\begin{proof}[Proof of Corollary \ref{cor:two proj for pinned dist}]{}
    The proof largely follows the argument of Theorem \ref{thm:two proj for general proj}. We give a brief sketch and highlight the necessary modifications.

    Assume $E$ is not purely $1$-unrectifiable. Then by Lemma \ref{lem:equiv def of unrect}, there exists a Lipschitz graph $\zeta\subseteq\mathbb R^2$ such that $\mathcal{H}^1\left(E\cap \zeta\right)>0$, with either $\zeta=\{(x,F(x)):x\in\mathbb R\}$ or $\zeta=\{(F(x),x):x\in\mathbb R\}$ for some Lipschitz function $F:\mathbb R\to\mathbb R$. We may assume $\zeta=\{(x,F(x)):x\in\mathbb R\}$. The other case is similar.

    Let $$A_1=\left\{x\in\mathbb R:(x,F(x))\in E\setminus \ell_{y,y'}\right\}$$ and $$A_2=\{x\in\mathbb R:(x,F(x))\in E\cap\ell_{y,y'}\}.$$ Then $A_1$ and $A_2$ are Borel sets, and it follows that $\mathcal{L}^1(A_1)+\mathcal{L}^1(A_2)>0$.

    If $\mathcal{L}^1(A_1)>0$. We choose $x_0\in A_1$ such that $x_0$ is a density point of $A_1$, and $F'(x_0)$ exists. Since $(x,F(x))\notin\ell_{y,y'}$,
    \[|((x_0,F(x_0))-y)\wedge ((x_0,F(x_0))-y')|\ne 0,\]
    which implies
    \[\text{either}\quad((x_0,F(x_0))-y)\cdot (1,F'(x_0)) \ne 0\quad\text{or}\quad ((x_0,F(x_0))-y')\cdot (1,F'(x_0))\ne 0.\]
    If $\mathcal{L}^1(A_1)=0$, $\mathcal{L}^1(A_2)>0$. We choose $x_0\in A_2$ such that $(x_0,F(x_0))\in \ell_{y,y'}\setminus\{y,y'\}$, $x_0$ is a density point of $A_2$, and $F'(x_0)$ exists. By Lemma \ref{lem:tangent vec of Lip along a line},
    \[\text{both}\quad((x_0,F(x_0))-y)\cdot (1,F'(x_0)) \ne 0\quad\text{and}\quad ((x_0,F(x_0))-y')\cdot (1,F'(x_0))\ne 0.\]
    In both cases, the argument proceeds as in the proof of Theorem \ref{thm:two proj for general proj}, which leads to a contradiction.
\end{proof}

\section{Alternative proof of Theorem \ref{thm:two proj for general proj}}\label{sec:contrapositive}


We provide an alternative proof of the contrapositive of the two-projection theorem for $1$-rectifiable sets. 
Besides giving a second proof of the classic two-projection theorem, the argument illustrates the geometric mechanism underlying the theorem: the interaction between projections and tangent directions. 

\begin{dfn}[Tangential Differentiability]{}
    Let $E\subseteq\mathbb R^n$ be a $k$-dimensional $C^1$-surface. A function $f:\mathbb R^n\to\mathbb R^m$ is \textit{tangentially differentiable} with respect to $E$ at $x$, if there exists a linear map $\nabla^E f(x):T_xE\to\mathbb R^m$ such that
    \[\lim_{t\to 0}\frac{f(x+tv)-f(x)}{t}=\nabla^E f(x)v,\]
    uniformly on $\{v\in T_x E :|v|=1\}$.

    The \textit{tangential Jacobian} of $f$ with respect to $E$ at $x$ is then defined as
    \[J^Ef(x)=\det(\nabla^E f(x)^*\nabla^E f(x))^{1/2}.\]
\end{dfn}

\begin{lem}[Rectifiability and Approximate Tangent Planes]{}\label{lem:approx tangent plane}
    Let $E\subseteq\mathbb R^2$ be a compact set with $\mathcal{H}^1(E)<\infty$. Then $E$ is $1$-rectifiable if and only if for $\mathcal{H}^1$-a.e.\ $x\in E$, there exists a unique approximate tangent $1$-plane $T_xE$ for $E$ at $x$.
\end{lem}

\begin{lem}[Rademacher-type Theorem for Rectifiable Sets]{}\label{lem:Rademacher type for rect}
    Let $E\subseteq\mathbb R^2$ be a $\mathcal{H}^1$-rectifiable compact set with $\mathcal{H}^1(E)<\infty$, and $f:\mathbb R^2 \to  \mathbb R$ be a Lipschitz function. Then $\nabla^E f(x)$ exists for $\mathcal{H}^1$-a.e.\ $x\in E$.
\end{lem}

\begin{prop}[Area Formula for Rectifiable Sets]{}\label{prop:area formula}
    Let $E\subseteq \mathbb R^2$ be a $1$-rectifiable set with $\mathcal{H}^1(E)<\infty$, and $f:\mathbb R^2\to\mathbb R$ be a Lipschitz function. Then
    \[\int_{E} J^Ef(x)\,d\mathcal{H}^1(x)=\int_{f(E)} \mathcal{H}^0\left(E\cap f^{-1}(y)\right)\,d\mathcal{H}^1(y).\]
\end{prop}

For an introduction to these results and their proofs, see \cite{MaggiBook}.

\begin{cor}[Contrapositive of Two-Projection Theorem for Generalized Projections]{}\label{cor:contrapositive}
    Let $\Lambda\subseteq\mathbb R^2$ be a Borel set, and $P:\mathbb R^2\times \Lambda\to \mathbb R$ be a continuous function such that, for each $\lambda\in \Lambda$, the map $P_\lambda(x):=P(x,\lambda)$ is Lipschitz on every compact subset of $\mathbb R^2$.

    Let $E\subseteq \mathbb R^2$ be a $1$-rectifiable compact set with $0<\mathcal{H}^1(E)<\infty$. Suppose that for each pair $\lambda,\lambda'\in \Lambda$ with $\lambda\ne\lambda'$, $\nabla^EP_\lambda(x)\ne \pm\nabla^EP_{\lambda'}(x)$ for $\mathcal{H}^1$-a.e.\ $x\in E$.
    Then $\mathcal{L}^1(P_\lambda(E))>0$ for all but at most one $\lambda\in \Lambda$.
\end{cor}

\begin{proof}
    Assume there exists distinct $\lambda,\lambda'\in \Lambda$ with $\mathcal{L}^1(P_{\lambda}(E))=\mathcal{L}^1(P_{\lambda'}(E))=0$.
    By Proposition \ref{prop:area formula}, we have
    \[\int_{E} J^E P_{\lambda}(x)\,d\mathcal{H}^1(x)=\int_{P_{\lambda}(E)} \mathcal{H}^0\left(E\cap P_{\lambda}^{-1}(y)\right)\,d\mathcal{H}^1(y).\]
    Moreover, by Lemma \ref{lem:approx tangent plane} and \ref{lem:Rademacher type for rect}, we calculate
    \[J^E P_{\lambda}(x)=|\nabla^EP_\lambda(x)v(x)|\quad\text{for $\mathcal{H}^1$-a.e.\ $x\in E$},\]
    where $v(x)\in T_xE$ is a unit tangent vector and $T_x E=\text{span}\{v(x)\}$.
    Since $\mathcal{L}^1(P_{\lambda}(E))=0$,
    we have
    \[\int_{P_{\lambda}(E)} \mathcal{H}^0\left(E\cap P_{\lambda}^{-1}(y)\right)\,d\mathcal{H}^1(y)=0.\]
    Thus,
    \[J^E P_{\lambda}(x)=|\nabla^EP_\lambda(x)v(x)|=0\quad\text{for $\mathcal{H}^1$-a.e.\ $x\in E$}.\]
    Similarly,
    \[J^E P_{\lambda'}(x)=|\nabla^E P_{\lambda'}(x)v(x)|=0\quad\text{for $\mathcal{H}^1$-a.e.\ $x\in E$}.\]
    Hence
    \[|\nabla^EP_\lambda(x)v(x)|=0=|\nabla^E P_{\lambda'}(x)v(x)|\quad\text{for $\mathcal{H}^1$-a.e.\ $x\in E$}.\]
    Since $\nabla^EP_\lambda(x)$ and $\nabla^EP_{\lambda'}(x)$ are linear maps,
    \[\nabla^EP_\lambda(x)=0=\nabla^E P_{\lambda'}(x)\quad\text{for $\mathcal{H}^1$-a.e.\ $x\in E$}.\]
    This contradicts the hypothesis that  $\nabla^EP_\lambda(x)\ne \pm\nabla^EP_{\lambda'}(x)$ for $\mathcal{H}^1$-a.e.\ $x\in E$.
\end{proof}

\begin{rmk}
Note that a higher dimensional variant of Corollary \ref{cor:contrapositive} is not immediate as the tangential Jacobian computation increases in complexity.
A higher dimensional variant, as mentioned above however, is obtained in \cite{BBMT26}.
\end{rmk}

\begin{rmk}
    If $P_\lambda$ is differentiable at $x\in E$, $J^EP_{\lambda}(x)=|\nabla P_\lambda(x)\cdot v(x)|$; this corresponds to Property (2) in Theorem \ref{thm:two proj for general proj}.
    Moreover, a weaker assumption is sufficient to conclude Corollary \ref{cor:contrapositive}: for each pair $\lambda,\lambda'\in\Lambda$ with $\lambda\ne\lambda'$, $\nabla^E P_\lambda(x)$ and $\nabla^E P_{\lambda'}(x)$ do not both vanish for $\mathcal{H}^1$-a.e.\ $x\in E$. That is,
    \[\mathcal{H}^1(\{x\in E:\nabla^E P_\lambda(x)=0=\nabla^E P_{\lambda'}(x)\})<\mathcal{H}^1(E).\]
\end{rmk}

 \section{Proof of Theorem \ref{thm:quant curve two}}\label{section:proof}

Our proof follows that of Tao's quantitative two-projection theorem in the linear setting with some necessary modifications. We first give the setup and outline the strategy for passing from qualitative to quantitative results.

\subsection{Setup}
If $N\lesssim 1$, the conclusion follows from the trivial bound $\mathcal{R}_{E,\text{res}}(r_N,1,N^{1/100})\leqslant 1$, so we may assume that $N$ is sufficiently large. For convenience, we assume that $N^{1/100}\in\mathbb N$ and $N^{1/100}\geqslant 3$.

Since $E\subseteq [0,1]^2$, it suffices to prove that for every Lipschitz function $F:\mathbb R\to \mathbb R$ with $\Lip(F)\leqslant N^{1/100}$,
\begin{equation}\label{eqn:case XY}
    \mathcal{L}^1(\{x\in[0,1]:\exists\,y\in[-r_N,r_N]\ \text{such that}\ (x,F(x)+y)\in E\})\lesssim N^{-1/100},
\end{equation}
and
\begin{equation}\label{eqn:case YX}
    \mathcal{L}^1(\{x\in[0,1]:\exists\,y\in[-r_N,r_N]\ \text{such that}\ (F(x)+y,x)\in E\})\lesssim N^{-1/100}.
\end{equation}
We prove (\ref{eqn:case XY}). The proof of (\ref{eqn:case YX}) is the same with some straightforward modifications. See Remark \ref{rmk:case YX}.

Fix such an $F$. Let
\[A:=\{x\in[0,1]:\exists\,y\in[-r_N,r_N]\ \text{such that}\ (x,F(x)+y)\in E\}.\]
Observe that $A$ is a compact set, because $E$ is compact.
Also, if $x\in A$, there exists $y\in[-r_N,r_N]$ such that $(x,F(x)+y)\in E$. Note that
\[|\Phi_\lambda(x,F(x)+y)-\Phi_\lambda(x,F(x))|=|F(x)+y+\gamma(\lambda-x)-F(x)-\gamma(\lambda-x)|=|y|\leqslant r_N,\]
so
\begin{equation}\label{eqn:alpha neighbor}
    \{\Phi_{\lambda}(x,F(x)):x\in A\}\subseteq \mathcal{N}_{r_N}(\Phi_{\lambda}(E)),\quad\text{for all}\;\;\lambda\in \Lambda.
\end{equation}

We will show (\ref{eqn:case XY}), that is,
\begin{equation*}
    \mathcal{L}^1(A)\lesssim N^{-1/100}
\end{equation*}
More precisely, we prove that if \eqref{eqn:case XY} fails, \eqref{eqn:alpha neighbor} implies that the small projection assumption fails for at least one of $\alpha$ and $\alpha'$.

\subsection{Proof strategy}
Recall that the two-projection theorem is proved by contradiction using the Lebesgue differentiation theorem, Rademacher's differentiation theorem, and intermediate value theorem. In the quantitative setting, we replace these pointwise statements with scale-dependent versions.

For the quantitative Lebesgue differentiation theorem, we select a dyadic interval that intersects $A$ and is stable across finer scales, so that the density remains high on all subintervals. Since the number of scales $\{r_n\}_{n=1}^N$ is finite, we use the pigeonhole principle to isolate a suitable scale. See Section \ref{sec:quant Lebesgue}.

A similar idea applies to the quantitative Rademacher's differentiation theorem. We approximate the projection map $x\mapsto \Phi_\lambda(x, F(x))$ by piecewise linear functions across scales, and use an $L^2$-estimate to obtain good approximate differentiability at a chosen scale. In particular, unlike in the linear case, we must not only quantify the derivative of the relevant Lipschitz function $F$ but also the function $\gamma$. See Section \ref{sec:quant Rademacher}.

Finally, we replace the intermediate value theorem with a counting argument. The main tool is the \textit{Rising Sun Lemma}, which allows us to track the total growth and decay on the well-behaved dyadic intervals and derive a contradiction with the small projection assumption. This step involves several technical difficulties that do not appear in the linear setting. See Section \ref{sec:rising sun}.

\subsection{Quantitative Lebesgue differentiation theorem}\label{sec:quant Lebesgue}
We first prove the quantitative Lebesgue differentiation theorem, which requires the following version of the pigeonhole principle. Note that it is also used in Section \ref{sec:quant Rademacher}.

\begin{lem}[Pigeonhole Principle, Lemma~1.14 in \cite{Tao09}]{}\label{pigeon}
    Let $(X,\mu)$ be a measure space. Suppose $E_0\subseteq E_1\subseteq\cdots\subseteq E_N$ is a sequence of measurable subsets of $X$ with $N\geqslant 2$. If $\varepsilon\in [1/N,1/2]$, there exists $n,m\in\{0,1\cdots,N\}$ with $m-n\geqslant \varepsilon N$ such that $\mu(E_m\setminus E_n)\lesssim\varepsilon\mu(E_N)$.
\end{lem}

\begin{proof}
    Let $j,k\in \{1,\cdots,N\}$. Then,
    $E_j\setminus E_{j-1}\subseteq E_{n+k}\setminus E_n$ if and only if $\max\{j-k,0\}\leqslant n\leqslant\min\{j-1,N-k\}.$
    Hence, $E_j\setminus E_{j-1}$ can belong to at most $k$ sets of the form $E_{n+k}\setminus E_n$. Note that
    \[E_{n+k}\setminus E_n=\bigcup\limits_{j=n+1}^{n+k} E_j\setminus E_{j-1}\text{ and so, } \mu(E_{n+k}\setminus E_n)=\sum\limits_{j=n+1}^{n+k} \mu(E_j\setminus E_{j-1}).\]
    It follows that
    \[\sum\limits_{n=0}^{N-k}\mu(E_{n+k}\setminus E_n)=\sum\limits_{n=0}^{N-k}\sum\limits_{j=n+1}^{n+k} \mu(E_j\setminus E_{j-1})\leqslant k\sum\limits_{j=1}^{N} \mu(E_j\setminus E_{j-1})\leqslant k\mu(E_N).\]
    By the pigeonhole principle, there exists $n\in \{0,1,\cdots,N-k\}$ such that
    \[\mu(E_{n+k}\setminus E_n)\leqslant \frac{k}{N-k}\mu(E_N).\]
    Setting $k=\left\lceil \varepsilon N\right\rceil$ gives the desired result.
\end{proof}

Now we state the quantitative Lebesgue differentiation theorem.


\begin{prop}[Quantitative Lebesgue Differentiation Theorem]{}
    For $n\in \{1,\cdots, N\}$, partition $[0,1]$ into dyadic intervals $\{[jr_n,(j+1)r_n]:j\in\mathbb N\}$ with length $r_n$, and let
    \[A_n:=\bigcup_{\substack{[jr_n,(j+1)r_n]\subseteq [0,1]\\ [jr_n,(j+1)r_n]\cap A\ne \varnothing}}[jr_n,(j+1)r_n].\]
    Then there exists an $n_0\in [0.1N,0.9N]\cap\mathbb N$ such that
    \begin{equation}\label{eqn:Delta A bound}
        \mathcal{L}^1\left(\Delta A\right)\lesssim N^{-3/100}
    \end{equation}
    where $\Delta A=A_{n_0-N^{-3/100}N}\setminus A_{n_0+N^{-3/100}N}$.
\end{prop}

\begin{proof}
    Since we assumed $N^{1/100}\geqslant 3$,  $n\pm N^{-3/100}N\in \{1,2,\cdots,N\}$ for every $n\in [0.1N,0.9N]\cap\mathbb N$. Note that
    \[A\subseteq A_N\subseteq\cdots\subseteq A_1\subseteq[0,1].\]
    The desired conclusion then follows from Lemma \ref{pigeon}.
\end{proof}

Informally, most points in the coarse-scale set $A_{n_0-N^{-3/100}N}$ remain density points at the finer scale $A_{n_0+N^{-3/100}N}$. Although $A_{n_0+N^{-3/100}N}$ is slightly larger than $A$, it still serves as a good approximation for $A$ in our analysis.

\subsection{Quantitative Rademacher's differentiation theorem}\label{sec:quant Rademacher}
Recall that $\Gamma=\{\left(t,\gamma(t)\right):t\in\mathbb R\}$. By the nonlinear nature of the curve projections, it is not sufficient to quantify the behavior of $F$ alone. We also need to quantify the curve $\gamma$ so that it admits a good linear approximation. The heuristic is to control the derivative of the function $x \mapsto \Phi_\lambda(x,F(x))$ for $\lambda\in \{\alpha,\alpha'\}$. This derivative takes the form $F'(x)-\gamma'(\lambda-x)$.

For $n\in\{1,2,\cdots,N\}$, recall the dyadic grid $\{jr_n\in [0,1]:j\in\mathbb N\}$. Let $F_n:[0,1]\to \mathbb R$ be the unique piecewise linear function that agrees with $F$ on this dyadic grid, that is, $F_n(jr_n)=F(jr_n)$, and is linear on each dyadic interval $[jr_n,(j+1)r_n]$. For each $\lambda\in\{\alpha,\alpha'\}$, let $\gamma_\lambda:[0,1]\to \mathbb R$ be defined as $\gamma_\lambda(x):=\gamma(\lambda-x)$. Similarly, for $\lambda\in\{\alpha,\alpha'\}$, let $\gamma_{\lambda,n}:[0,1]\to \mathbb R$ be the unique piecewise linear function that agrees with $\gamma_\lambda$ on this dyadic grid.

We start with a property of the proceeding piecewise linear functions that will be used in Section \ref{sec:contradiction}.
\begin{lem}[Linear Approximation Difference]{}\label{lem:L infty difference}
For every $n\in\{1,\cdots, N\}$,
\begin{enumerate}[\rm(1)]
    \item $\|F-F_n\|_{L^\infty([0,1])}\leqslant N^{1/100}r_n$, and
    \item $\|\gamma_\lambda-\gamma_{\lambda,n}\|_{L^\infty([0,1])}\leqslant r_n$ for $\lambda\in\{\alpha,\alpha'\}$.
\end{enumerate}
\end{lem}

\begin{proof}
    Let $x\in (jr_n,(j+1)r_n)$ for some $j\in\mathbb N$. Note that either
    \[|F(x)-F_n(x)|\leqslant |F(x)-F_n(jr_n)|\ \text{or}\ |F(x)-F_n(x)|\leqslant |F(x)-F_n((j+1)r_n)|.\]
    Without loss of generality, we may assume that $|F(x)-F_n(x)|\leqslant |F(x)-F_n(jr_n)|$. Since $\Lip(F)\leqslant N^{1/100}$,
    \begin{align*}
    |F(x)-F_n(x)|&\leqslant |F(x)-F_n(jr_n)|=|F(x)-F(jr_n)|\\
    &\leqslant N^{1/100}|x-jr_n|\leqslant N^{1/100}r_n,
    \end{align*}
    so $\|F-F_n\|_{L^\infty([0,1])}\leqslant N^{1/100}r_n$. The proof of (2) is similar.
\end{proof}

Since $\Lip(F)\leqslant N^{1/100}$ and $\sup\limits_{t\in \mathbb R}|\gamma'(t)|\leqslant 1$, we have $\Lip(F_n)\leqslant N^{1/100}$ and $\Lip(\gamma_{\lambda,n})\leqslant 1$ for $\lambda\in\{\alpha,\alpha'\}$. Also, since $F_n'$ exists a.e., $\|F_n'\|_{L^\infty([0,1])}\leqslant N^{1/100}$. This gives $\|F_n'\|_{L^2([0,1])}^2\lesssim N^{2/100}$. Similarly, $\|\gamma_{\lambda,n}'\|_{L^2([0,1])}^2\leqslant 1\leqslant N^{2/100}$ for $\lambda\in\{\alpha,\alpha'\}$. Thus,
\begin{equation}\label{eqn:L2 sum}
    \|F_n'\|_{L^2([0,1])}^2+\|\gamma_{\alpha,n}'\|_{L^2([0,1])}^2+\|\gamma_{\alpha',n}'\|_{L^2([0,1])}^2\lesssim N^{2/100}.
\end{equation}

We now show that $\|F_n'\|_{L^2([0,1])}^2,\ \|\gamma_{\alpha,n}'\|_{L^2([0,1])}^2$, and $\|\gamma_{\alpha',n}'\|_{L^2([0,1])}^2$ are increasing functions in $n$. To see this, we first prove $F_{k+1}'-F_k'$, $\gamma_{\alpha,k+1}'-\gamma_{\alpha,k}'$, and $\gamma_{\alpha',k+1}'-\gamma_{\alpha',k}'$ are respectively pairwise orthogonal in the sense of $L^2([0,1])$.
\begin{lem}[Orthogonality]{}\label{lem:orthogonal}
    In $L^2([0,1])$, for every pair $(n,k)\in\{1,\cdots,N-1\}^2$ with $n\ne k$,
    \begin{enumerate}[\rm(1)]
        \item $\langle F_{n+1}'-F_n',F_{k+1}'-F_k'\rangle=0$ and $\langle F_{n+1}'-F_n',F_1'\rangle=0$;
        \item $\langle\gamma_{\lambda,n+1}'-\gamma_{\lambda,n}',\gamma_{\lambda,k+1}'-\gamma_{\lambda,k}'\rangle=0$ and $\langle\gamma_{\lambda,n+1}'-\gamma_{\lambda,n}',\gamma_{\lambda,1}'\rangle=0$ for $\lambda\in\{\alpha,\alpha'\}$.
    \end{enumerate}
\end{lem}

\begin{proof}
    Without loss of generality, we may assume $n>k$. By the fundamental theorem of calculus,
    \begin{align*}
        \int_{jr_n}^{(j+1)r_n} F_{n+1}'-F_n'\,d\mathcal{L}^1&=F_{n+1}((j+1)r_n)-F_n((j+1)r_n)-F_{n+1}(jr_n)+F_n(jr_n)\\
        &=F((j+1)r_n)-F((j+1)r_n)-F(jr_n)+F(jr_n)\\
        &=0.
    \end{align*}
    Since $n>k$, $r_n\leqslant r_{k+1}<r_k$, and thus $F_{k+1}'-F_k'$ is constant on the dyadic interval of length $r_n$. Therefore,
    \begin{align*}
    \int_0^1 \left(F_{n+1}'-F_n'\right)&\left(F_{k+1}'-F_k'\right) \,d\mathcal{L}^1=\sum_{j\in\mathbb N:jr_n\in [0,1]}\int_{jr_n}^{(j+1)r_n} \left(F_{n+1}'-F_n'\right)\left(F_{k+1}'-F_k'\right) \,d\mathcal{L}^1=0.\end{align*}
  Similarly, $\langle F_{n+1}'-F_n',F_1'\rangle=0$. The proof of (2) follows.
\end{proof}

\begin{cor}[Increasing Functions]{}\label{cor:increasing functions}
    $\|F_n'\|_{L^2([0,1])}^2,\ \|\gamma_{\alpha,n}'\|_{L^2([0,1])}^2$, and $\|\gamma_{\alpha',n}'\|_{L^2([0,1])}^2$ are increasing functions of $n$.
\end{cor}

\begin{proof}
    Since $F_n'=F_1'+\sum\limits_{k=1}^{n-1} (F_{k+1}'-F_k')$, by Lemma \ref{lem:orthogonal},
    \[\|F_n'\|_{L^2([0,1])}^2=\|F_1'\|_{L^2([0,1])}^2+\sum\limits_{k=1}^{n-1} \|F_{k+1}'-F_k'\|_{L^2([0,1])}^2,\]
    so $\|F_n'\|_{L^2([0,1])}^2$ is an increasing function of $n$. Similarly, $\|\gamma_{\lambda,n}'\|_{L^2([0,1])}^2=\|\gamma_{\lambda,1}'\|_{L^2([0,1])}^2+\sum\limits_{k=1}^{n-1} \|\gamma_{\lambda,k+1}'-\gamma_{\lambda,k}'\|_{L^2([0,1])}^2$ for $\lambda\in\{\alpha,\alpha'\}$.
\end{proof}


Now we are ready to state the Quantitative Rademacher's Differentiation Theorem.
\begin{prop}[Quantitative Rademacher's Differentiation Theorem]{}
    There exists an $n_1\in [n_0-0.9N^{-3/100}N,n_0+0.9N^{-3/100}N]\cap\mathbb N$ such that
    \begin{equation}\label{eqn:L2 derivative diff}
    \left\|\Delta F'\right\|_{L^2([0,1])}^2+\left\|\Delta \gamma_{\alpha}'\right\|_{L^2([0,1])}^2+\left\|\Delta \gamma_{\alpha'}'\right\|_{L^2([0,1])}^2\lesssim N^{-4/100},
    \end{equation}
    where $\Delta F'=F_{n_1+N^{-9/100}N}'-F_{n_1-N^{-9/100}N}'$, $\Delta \gamma_{\alpha}'=\gamma_{\alpha,n_1+N^{-9/100}N}'-\gamma_{\alpha,n_1-N^{-9/100}N}'$, and $\Delta \gamma_{\alpha'}'=\gamma_{\alpha',n_1+N^{-9/100}N}'-\gamma_{\alpha',n_1-N^{-9/100}N}'$.
\end{prop}

\begin{proof}
    Since we assumed $N^{1/100}\geqslant 3$, $n\pm N^{-9/100}N\in\{1,2,\cdots,N\}$ for every $n\in [n_0-0.9N^{-3/100}N,n_0+0.9N^{-3/100}N]\cap\mathbb N$. Then by the pigeonhole principle, Corollary \ref{cor:increasing functions}, and (\ref{eqn:L2 sum}), there exists $n_1\in [n_0-0.9N^{-3/100}N,n_0+0.9N^{-3/100}N]$ such that
\begin{align*}
   & \left\|F_{n_1+N^{-9/100}N}'\right\|_{L^2([0,1])}^2-\left\|F_{n_1-N^{-9/100}N}'\right\|_{L^2([0,1])}^2\hspace*{15mm}\\
    &+\left\|\gamma_{\alpha,n_1+N^{-9/100}N}'\right\|_{L^2([0,1])}^2-\left\|\gamma_{\alpha,n_1-N^{-9/100}N}'\right\|_{L^2([0,1])}^2\\
    &+\left\|\gamma_{\alpha',n_1+N^{-9/100}N}'\right\|_{L^2([0,1])}^2-\left\|\gamma_{\alpha',n_1-N^{-9/100}N}'\right\|_{L^2([0,1])}^2\\
    &\lesssim N^{2/100}(N^{-3/100}N)^{-1}N^{-9/100}N\\
    &=N^{-4/100}.
    \end{align*}
    By Lemma \ref{lem:orthogonal}, $\langle \Delta F',F_{n_1-N^{-9/100}N}'\rangle=\langle \Delta \gamma_{\lambda}',\gamma_{\lambda,n_1-N^{-9/100}N}'\rangle=0$ for $\lambda\in\{\alpha,\alpha'\}$, so we have (\ref{eqn:L2 derivative diff}).
\end{proof}

Informally, this shows that $F$ has good approximate differentiability between the coarse scale $n_1 - N^{-9/100}N$ and the finer scale $n_1 + N^{-9/100}N$. This serves as a quantitative version of Rademacher's differentiation theorem. We also impose quantitative control on $\gamma_{\alpha}$ and $\gamma_{\alpha'}$ so that they interact well with the behavior of $F$.

In the rest of paper, we abbreviate
\[n_0^{\pm}=n_0\pm N^{-3/100}N\ \text{and}\  n_1^{\pm}=n_1\pm N^{-9/100}N.\]

\subsection{Rising Sun Lemma}\label{sec:rising sun}
In this section, we highlight the importance of the Rising Sun Lemma. Roughly speaking, the Rising Sun Lemma serves as a substitute for the intermediate value theorem. It allows us to track the cumulative growth and decay of the map $\lambda\mapsto\Phi_\lambda(x,F(x))$, for $\lambda\in\{\alpha,\alpha'\}$, on suitably chosen dyadic intervals where the functions $A$, $F$, and $\gamma$ are well behaved. The key point is that, on each such interval, the linear approximation has well-controlled slope. The Rising Sun Lemma then guarantees that either only few intervals contribute to a given level set, or a positive proportion of the contributing intervals exhibit unfavorable behavior.

We first fix a dyadic interval $I_0\subseteq A_{n_0^-}$ with length $r_{n_1^-}$ such that the following two inequalities hold:
\begin{equation}\label{eqn:Delta A inter I0}
    \mathcal{L}^1\left(\Delta A\cap I_0\right)\lesssim N^{-2/100}\mathcal{L}^1(I_0),
\end{equation}
and
\begin{equation}\label{eqn:L2 sum on I0}
    \int_{I_0} \left|F_{n_1^+}'-F_{n_1^-}'\right|^2+\left|\gamma_{\alpha,n_1^+}'-\gamma_{\alpha,n_1^-}'\right|^2+\left|\gamma_{\alpha',n_1^+}'-\gamma_{\alpha',n_1^-}'\right|^2\,d\mathcal{L}^1\lesssim N^{-3/100}\mathcal{L}^1\left(I_0\right).
\end{equation}
We postpone the process to find such $I_0$ until Section \ref{sec:contradiction}.

Write $I_0=[a,b]$. Note that $F_{n_1^-}'$ is constant on $I_0$. Let $c=F_{n_1^-}'$ on $I_0$. Since $\gamma'$ is injective, from the hypotheses on $\alpha,\alpha'$, we have either 
\[|c-\gamma'(\alpha-a)|\gtrsim 1\text{ or }|c-\gamma'(\alpha'-a)|\gtrsim 1.\] Without loss of generality, we may assume $|c-\gamma'(\alpha-a)|\gtrsim 1$. Let
\begin{equation}\label{eqn:fix point derivative}
    t=\nabla\Phi_\alpha(a,F(a))\cdot (1,c)=c-\gamma'(\alpha-a)
\end{equation}

\subsubsection{The case of $t>0$}
Suppose $t>0$ and partition $I_0$ into $S=r_{n_1^-}/r_{n_1^+}$ dyadic intervals $J_1,J_2,\cdots,J_S$ of length $r_{n_1^+}$. Also, partition $\mathbb R$ into disjoint intervals $\{K_m:m\in\mathbb Z\}$ of length $100N^{1/100}r_{n_1^+}$.

\begin{dfn}[``Reach'' Definition 1]
    We say that a dyadic interval $J_j$ \textit{reaches} an interval $K_m$ if there exists at least one $x\in J_j$ such that $F_{n_1^+}(x)+\gamma_{\alpha,n_1^+}(x)\in K_m$.
\end{dfn}

\begin{lem}[$J_j$'s Maximal Reach Number]\label{lem:maximal number reach}
    For each $J_j$, there exists a unique $K_m$ such that
    \begin{enumerate}[\rm(1)]
        \item either $J_j$ only reaches the interval $K_m$, or
        \item $J_j$ reaches both intervals $K_m$ and $K_{m+1}$.
    \end{enumerate}
\end{lem}

\begin{proof}
    Let $x_1,x_2\in J_j$. Then
    \begin{align*}
    \bigl|F_{n_1^+}(x_1)+\gamma_{\alpha,n_1^+}(x_1)&-F_{n_1^+}(x_2)-\gamma_{\alpha,n_1^+}(x_2)\bigr|\leqslant \bigl|F_{n_1^+}(x_1)-F_{n_1^+}(x_2)\bigr|+\bigl|\gamma_{\alpha,n_1^+}(x_1)-\gamma_{\alpha,n_1^+}(x_2)\bigr|.
    \end{align*}
    Since $\Lip(F_{n_1^+})\leqslant N^{1/100}$, $\Lip(\gamma_{\alpha,n_1^+})\leqslant 1$, and $|x_1-x_2|\leqslant\mathcal{L}^1(J_j)=r_{n_1^+}$, we have
      \begin{align*}
      \left|F_{n_1^+}(x_1)+\gamma_{\alpha,n_1^+}(x_1)-F_{n_1^+}(x_2)-\gamma_{\alpha,n_1^+}(x_2)\right|
      &\leqslant(N^{1/100}+1)|x_1-x_2|\\&\leqslant100N^{1/100}r_{n_1^+}.\qedhere
      \end{align*}
\end{proof}

\begin{dfn}[Good and Bad Interval Definition 1]
    A dyadic interval $J_j$ is said to be \textit{good} if
    \[\left|  F'_{n_1^+}(x)+\gamma'_{\alpha,n_1^+}(x)-t\right|<\frac{t}{100},\quad\text{for all}\;\;x\in J_j,\]
    and \textit{bad} otherwise. Since $F_{n_1^+}'+\gamma_{\alpha,n_1^+}'$ is constant on $J_j$, we will drop the variable $x$ in practice.
\end{dfn}

Next, we show that there are few bad intervals.

\begin{lem}[Upper Bound for the Number of Bad Intervals]\label{lem:upper bound on bad 1}
    At most $O(N^{-3/100}S)$ of the intervals $J_j$ are bad.
\end{lem}

\begin{proof}
    We prove that
    \[\mathcal{L}^1\left(\left\{x\in I_0:\left|F_{n_1^+}'+\gamma_{\alpha,n_1^+}'-t\right|\geqslant \frac{t}{100}\right\}\right)\lesssim N^{-3/100}\mathcal{L}^1(I_0).\]
    By the Chebyshev's inequality,
    \[\mathcal{L}^1\left(\left\{x\in I_0:\left|F_{n_1^+}'+\gamma_{\alpha,n_1^+}'-t\right|\geqslant  \frac{t}{100}\right\}\right)\leqslant \frac{1}{(t/100)^2}\int_{I_0} \left|F_{n_1^+}'-F_{n_1^-}'+\gamma_{\alpha,n_1^+}'+\gamma'(\alpha-a)\right|^2\,d\mathcal{L}^1.\]
    By the triangle inequality, we further decompose
    \begin{align*}
        &\int_{I_0} \left|F_{n_1^+}'-F_{n_1^-}'+\gamma_{\alpha,n_1^+}'+\gamma'(\alpha-a)\right|^2\,d\mathcal{L}^1\\&      
        \lesssim \int_{I_0} \left|F_{n_1^+}'-F_{n_1^-}'\right|^2\,d\mathcal{L}^1+\hspace{-.5mm}\int_{I_0} \left|\gamma_{\alpha,n_1^+}'-\gamma_{\alpha,n_1^-}'\right|^2 \,d\mathcal{L}^1+\hspace{-.5mm}\int_{I_0} \left|\gamma_{\alpha,n_1^-}'+\gamma'(\alpha-a)\right|^2 \,d\mathcal{L}^1.
    \end{align*}
    By (\ref{eqn:L2 sum on I0}), we have
    \begin{align*}
        \int_{I_0} \left|F_{n_1^+}'-F_{n_1^-}'\right|^2\,d\mathcal{L}^1+\int_{I_0} \left|\gamma_{\alpha,n_1^+}'-\gamma_{\alpha,n_1^-}'\right|^2\,d\mathcal{L}^1\lesssim N^{-3/100}\mathcal{L}^1(I_0).
    \end{align*}
    Thus, it suffices to prove
    \[\int_{I_0} \left|\gamma_{\alpha,n_1^-}'+\gamma'(\alpha-a)\right|^2\,d\mathcal{L}^1 \lesssim N^{-3/100}\mathcal{L}^1(I_0).\]
    Note that $\gamma_{\alpha,n_1^-}'$ is constant on $I_0$, so it remains to show that
    \[\left|\gamma_{\alpha,n_1^-}'+\gamma'(\alpha-a)\right|^2\lesssim N^{-3/100}.\]

    First, we require a decay bound at the scale $n_1 - N^{-9/100}N$. As $N^{1/100}\geqslant 3$,
    \begin{align*}
        n_1^-&=n_1-N^{-9/100}N\\
        &\geqslant  n_0-0.9N^{-3/100}N-N^{-9/100}N\\
        &\geqslant  0.1N(1-9N^{-3/100}-10N^{-9/100})\\
        &\geqslant\frac{1}{20}N.
    \end{align*}
    Also, since $r_{n+1}\leqslant \frac{1}{2}r_n$,
    \[\mathcal{L}^1(I_0)=r_{n_1^-}\leqslant 2^{-(n_1-N^{-9/100}N)}\leqslant 2^{-\frac{1}{20}N}.\]
    By the fundamental theorem of calculus,
    \[\left|\gamma_{\alpha,n_1^-}'+\gamma'(\alpha-a)\right|^2=\left|\frac{\gamma(\alpha-b)-\gamma(\alpha-a)}{\mathcal{L}^1(I_0)}+\gamma'(\alpha-a)\right|^2=\left|\frac{1}{\mathcal{L}^1(I_0)}\int_{I_0}\gamma'(\alpha-x)-\gamma'(\alpha-a)\,d\mathcal{L}^1(x)\right|^2.\]
    Since $\gamma'$ is $M$-bi-Lipschitz, so
    \begin{align*}
        \left|\frac{1}{\mathcal{L}^1(I_0)}\int_{I_0}\gamma'(\alpha-x)-\gamma'(\alpha-a)\,d\mathcal{L}^1(x)\right|^2
        &\leqslant \left(\frac{1}{\mathcal{L}^1(I_0)}\int_{I_0}M(x-a)\,d\mathcal{L}^1(x)\right)^2\\
        &=2^{-2}M^2\mathcal{L}^1(I_0)^2\\
        &=2^{-2}M^2r_{n_1^-}^2\\
        &\leqslant 2^{-2(n_1-N^{-9/100}N+1)}M^2\\
        &\leqslant 2^{-\frac{1}{10}N} M^2.
    \end{align*}
    Since we assume that $N$ is sufficiently large, we can further assume that
    \[2^{-\frac{1}{10}N}M^2\leqslant N^{-3/100}.\]
    This gives that
    \[\int_{I_0} \left|\gamma_{\alpha,n_1^-}'+\gamma'(\alpha-a)\right|^2\,d\mathcal{L}^1\lesssim N^{-3/100}\mathcal{L}^1(I_0).\]

    Therefore,
    \[\mathcal{L}^1\left(\left\{x\in I_0:\left|F_{n_1^+}'+\gamma_{\alpha,n_1^+}'-t\right|\geqslant  \frac{t}{100}\right\}\right)\lesssim N^{-3/100}\mathcal{L}^1(I_0).\]
    Note that the left hand side is $r_{n_1^+}(\#\ \text{of bad $J_j$'s})$, while the right hand side is $N^{-3/100}r_{n_1^+}S$.
    This implies that at most $O(N^{-3/100}S)$ of the intervals $J_1,J_2,\cdots,J_S$ are bad.
\end{proof}

Before stating the Rising Sun Lemma, we need one final lemma to compare the length of $J_j$ to its image, under the assumption that $J_j$	is either good or bad.

\begin{lem}[Curve projection is increasing on good interval]\label{lem:increase on good interval}
    Fix $J_j$.
    \begin{enumerate}[\rm(1)]
        \item If $J_j$ be a good interval, the function $x\mapsto F_{n_1^+}(x)+\gamma_{\alpha,n_1^+}(x)$ is increasing on $J_j$. Moreover, its slope is comparable to $t$.
        \item If $J_j$ is a bad interval, the function $x\mapsto F_{n_1^+}(x)+\gamma_{\alpha,n_1^+}(x)$ is still linear, but the slope can have either sign and, by the proof of Lemma \ref{lem:maximal number reach}. has magnitude $O(N^{1/100}).$
    \end{enumerate}
\end{lem}

\begin{proof}
    The derivative of the function $x\mapsto F_{n_1^+}(x)+\gamma_{\alpha,n_1^+}(x)$ is $F_{n_1^+}'+\gamma_{\alpha,n_1^+}'$.
    Since $J_j$ is a good interval, $\left|F_{n_1^+}'+\gamma'_{\alpha,n_1^+}-t\right|<\frac{t}{100}$ on $J_j$. This gives the two following bounds:
    \begin{align*}
        F_{n_1^+}'+\gamma_{\alpha,n_1^+}'=F_{n_1^+}'+\gamma_{\alpha,n_1^+}'-t+t\geqslant -\frac{t}{100}+t\sim t,
    \end{align*}
    and
    \[F_{n_1^+}'+\gamma_{\alpha,n_1^+}'=F_{n_1^+}'+\gamma_{\alpha,n_1^+}'-t+t\leqslant \frac{t}{100}+t\sim t.\qedhere\]
\end{proof}

Now we state the Rising Sun Lemma.
\begin{lem}[Rising Sun Lemma 1]\label{rising sun 1}
    Fix $K_m$. Then at least one of the following statements is true:
    \begin{enumerate}[\rm(1)]
        \item There are at most $O(N^{1/100})$ intervals $J_j$ that reach $K_m$.
        \item Of all the intervals $J_j$ that reach $K_m$, the proportion of those $J_j$ that are bad is $\gtrsim N^{-1/100}$.
    \end{enumerate}
\end{lem}

\begin{proof}
    Since the function $x\mapsto F_{n_1^+}(x)+\gamma_{\alpha,n_1^+}(x)$ is linear and its derivative is positive on good intervals $J_j$, $\{F_{n_1^+}(x)+\gamma_{\alpha,n_1^+}(x):x\in J_j\}$ is a nonempty interval and by Lemma \ref{lem:increase on good interval},
    \[\mathcal{L}^{1}(\{F_{n_1^+}(x)+\gamma_{\alpha,n_1^+}(x):x\in J_j\})\sim_t \mathcal{L}^1(J_j).\]
    Since $K_m$ has length $100N^{1/100}r_{n_1^+}$, any consecutive string $J_{j+1},J_{j+2},\cdots,J_{j+p}$ of good intervals that reach $K_m$ can have length at most $p=O(N^{1/100})$.

    Now consider a maximal consecutive string $\Upsilon=\{J_{j+1},J_{j+2},\cdots,J_{j+q}\}$ that reach $K_m$, that is, $J_j$ and $J_{q+1}$ do not reach $K_m$. We split into two cases:
 
        
        \begin{description}
\item[\underline{Case 1.}]$N^{1/100}\ll q$, that is, $\Upsilon$ has length much larger than $N^{1/100}$.

        If $\Upsilon$ has $B$ bad intervals, there are $q-B$ good intervals. Note that a consecutive string of good intervals can at most have length $O(N^{1/100})$. So in the worst case, between any two bad intervals, there can be at most $O(N^{1/100})$ good intervals. This gives that
\begin{gather*}
q-B\lesssim (B+1)N^{1/100}, \hspace{3mm}\text{which implies,}\\
 B\geqslant  \frac{q-CN^{1/100}}{CN^{1/100}+1}\gtrsim \frac{q-CN^{1/100}}{CN^{1/100}}=\frac{q}{CN^{1/100}}-1,
 \end{gather*}
        where $C>0$ is the implicit constant. Since $N^{1/100}\ll q$,
        \[\frac{B}{q}\gtrsim \frac{1}{CN^{1/100}}-\frac{1}{q}\gtrsim N^{-1/100}.\]
        Thus, at least $\gtrsim N^{-1/100}$ of the intervals in $\Upsilon$ must be bad.
\item[\underline{Case 2.}]    $q=O(N^{1/100})$.   
        \begin{enumerate}[(i)]
            \item There is at least one bad interval. Then clearly $\gtrsim N^{-1/100}$ of the intervals in $\Upsilon$ are bad.
            \item $\Upsilon$ consists entirely of good intervals.
        \end{enumerate}
\end{description}

    Call a consecutive string that has length $O(N^{1/100})$ and consists entirely of good intervals that reach $K_m$ an ``exceptional string''; that is, in Case 2 (ii). Note that for every exceptional string, the function $x\mapsto F_{n_1^+}(x)+\gamma_{\alpha,n_1^+}(x)$ increases monotonically from below $K_m$ to above $K_m$. Thus, by the intermediate value theorem, between any two exceptional strings there must be at least one bad interval that reaches $K_m$, so
    \[\#\ \text{of exceptional strings}\leqslant \#\ \text{of bad intervals that reach } K_m+1.\]
    On the other hand, by Cases 1 and 2 (i), on all the non-exceptional strings, $\gtrsim N^{-1/100}$ of the intervals are bad. Since every exceptional string has length $O(N^{1/100})$, the claim follows.
\end{proof}

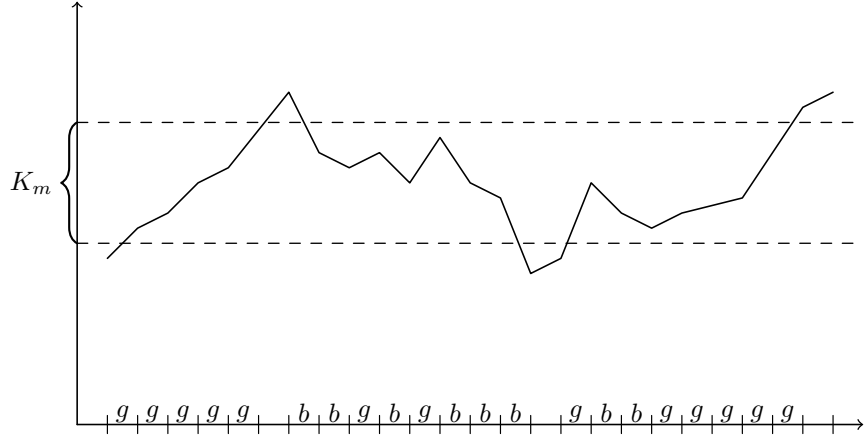
\begin{figure}[H]
\centering
\begin{tikzpicture}[x=0.4cm,y=0.4cm,line cap=round,line join=round]

\draw[->,line width=0.6pt] (0,0) -- (26,0);
\draw[->,line width=0.6pt] (0,0) -- (0,14);

\draw[dashed,dash pattern=on 4pt off 4pt,line width=0.5pt] (0,10) -- (26,10);
\draw[dashed,dash pattern=on 4pt off 4pt,line width=0.5pt] (0,6)  -- (26,6);

\draw[decorate,decoration={brace,amplitude=6pt}, thick] (0,6) -- (0,10);
\node[left] at (-0.5,8) {$K_m$};

\foreach \i in {1,...,25} {
  \draw[line width=0.45pt] (\i,0.3) -- (\i,-0.3);
}

\node at (1.5,0.35) {$g$};
\node at (2.5,0.35) {$g$};
\node at (3.5,0.35) {$g$};
\node at (4.5,0.35) {$g$};
\node at (5.5,0.35) {$g$};

\node at (7.5,0.4) {$b$};
\node at (8.5,0.4) {$b$};
\node at (9.5,0.35) {$g$};

\node at (10.5,0.4) {$b$};
\node at (11.5,0.35) {$g$};
\node at (12.5,0.4) {$b$};
\node at (13.5,0.4) {$b$};
\node at (14.5,0.4) {$b$};

\node at (16.5,0.35) {$g$};
\node at (17.5,0.4) {$b$};
\node at (18.5,0.4) {$b$};
\node at (19.5,0.35) {$g$};
\node at (20.5,0.35) {$g$};
\node at (21.5,0.35) {$g$};
\node at (22.5,0.35) {$g$};
\node at (23.5,0.35) {$g$};

\draw[line width=0.6pt]
  plot coordinates {
    (1,5.5) (2,6.5) (3,7) (4,8) (5,8.5)
    (7,11) (8,9) (9,8.5) (10,9) (11,8)
    (12,9.5) (13,8) (14,7.5) (15,5) (16,5.5)
    (17,8) (18,7) (19,6.5) (20,7) (22,7.5) (24,10.5)
    (25,11)
  };

\end{tikzpicture}
\caption{This is a graph of the function $x\mapsto F_{n_1^+}(x)+\gamma_{\alpha,n_1^+}(x)$. We distinguish the intervals $J_j$ that reach an interval $K_m$. The good intervals $J_j$ that reach $K_m$ are marked with “$g$”, while the bad intervals are marked with “$b$”. Note that the first string is an exceptional string.}
\end{figure}

\subsubsection{The case of $t<0$}
Next, we consider the case $t<0$. Recall the definition of $t$ in (\ref{eqn:fix point derivative}). The argument differs slightly from the case $t>0$. It relies on the geometry of $\Gamma$, which implicitly exhibits the \textit{transversality} of the curve projections.

Recall we partition $I$ into $S=r_{n_1^-}/r_{n_1^+}$ dyadic intervals $J_1,J_2,\cdots,J_S$ of length $r_{n_1^+}$. We also partition $\mathbb R$ into disjoint intervals $\{K_m:m\in\mathbb Z\}$ of length $100N^{1/100}r_{n_1^+}$.

\begin{dfn}[``Reach'' Definition 2]
    We say a dyadic interval $J_j$ \textit{reaches} an interval $K_m$ if there exists at least one $x\in J_j$ with $\Phi_\alpha(x,F_{n_1^+}(x))\in K_m$.
\end{dfn}

\begin{lem}[$J_j$'s Maximal Reach Number]
    For each $J_j$, there exists a unique $K_m$ such that
    \begin{enumerate}[\rm(1)]
        \item either $J_j$ only reaches the interval $K_m$, or
        \item $J_j$ reaches both intervals $K_m$ and $K_{m+1}$.
    \end{enumerate}
\end{lem}

\begin{proof}
    Let $x_1,x_2\in J_j$. Following Definition \ref{dfn:curve proj on cube}, we have
    \[|\Phi_\alpha(x_1,F_{n_1^+}(x_1))-\Phi_\alpha(x_2,F_{n_1^+}(x_2))|\leqslant (N^{1/100}+1)|x_1-x_2|\leqslant 100N^{1/100}r_{n_1^+}.\qedhere\]
\end{proof}

\begin{dfn}[Good and Bad Interval Definition 2]
    A dyadic interval $J_j$ is said to be \textit{good} if
    \[\left|F_{n_1^+}'-c\right|<-\frac{t}{100},\]
    and \textit{bad} otherwise.
\end{dfn}

\begin{lem}[Upper Bound for the Number of Bad Intervals]
    At most $O(N^{-3/100}S)$ of the intervals $J_j$ are bad.
\end{lem}

\begin{proof}
    The proof is similar to Lemma \ref{lem:upper bound on bad 1} by using Chebyshev's inequality.
\end{proof}

\begin{lem}[Curve projection is decreasing on good interval]\label{lem:decrease on good}
    Let $J_j$ be a good interval. Then the function $x\mapsto \Phi_\alpha(x,F_{n_1^+}(x))$ is decreasing on $J_j$.
\end{lem}

\begin{proof}
    We compute the derivative of the function $x\mapsto \Phi_\alpha(x,F_{n_1^+}(x))$ as $F_{n_1^+}'-\gamma'(\alpha-x)$.
    Since $\gamma'$ is monotone decreasing, $x\mapsto F_{n_1^+}'-\gamma'(\alpha-x)$ is also monotone decreasing on $J_j$. Moreover, $\gamma'(\alpha-a)-\gamma'(\alpha-x)\leqslant 0$ for every $x\in J_j$.
    Thus,
    \[F_{n_1^+}'-\gamma'(\alpha-x)=F_{n_1^+}'-c+c-\gamma'(\alpha-a)+\gamma'(\alpha-a)-\gamma'(\alpha-x)\leqslant -\frac{t}{100}+t\sim t.\qedhere\]
\end{proof}

\begin{lem}[Oscillation Control on Good Intervals]\label{lem:good osc control for curve proj}
    Fix $J_j$.
    \begin{enumerate}[\rm(1)]
        \item If $J_j$ is a good interval,
        \[\max\limits_{x,y\in J_j} \left|\Phi_\alpha(x,F_{n_1^+}(x))-\Phi_\alpha(y,F_{n_1^+}(y))\right|\sim_t\mathcal{L}^1(J_j).\]
        \item If $J_j$ is a bad interval, the function $x\mapsto \Phi_\alpha(x,F_{n_1^+}(x))$ is still concave down, but the slope can have either sign and has magnitude $O(N^{1/100})$.
    \end{enumerate}
\end{lem}

\begin{proof}
    Write $J_j=[p_j,q_j]$. By Lemma \ref{lem:decrease on good}, the function $x\mapsto \Phi_\alpha(x,F_{n_1^+}(x))$ is decreasing on $J_j$, so it suffices to prove that
    \[\Phi_\alpha(p_j,F_{n_1^+}(p_j))-\Phi_\alpha(q_j,F_{n_1^+}(q_j))\sim_t q_j-p_j.\]
    Since $J_j$ is a good interval, on $J_j$, $\left|F_{n_1^+}'-c\right|<-\frac{t}{100}$ which implies that $\left|F_{n_1^+}'\right|<-\frac{t}{100}+|c|.$
    Thus,
    \begin{align*}
        \Phi_\alpha(p_j,F_{n_1^+}(p_j))-\Phi_\alpha(q_j,F_{n_1^+}(q_j))&=F_{n_1^+}(p_j)-F_{n_1+}(q_j)+\gamma(\alpha-p_j)-\gamma(\alpha-q_j)\\
        &\leqslant\left|F_{n_1^+}'\right|(q_j-p_j)+(q_j-p_j)\\
        &\leqslant \left(-\frac{t}{100}+|c|+1\right)(q_j-p_j).
    \end{align*}

    On the other hand, it follows from Lemma \ref{lem:decrease on good}  that
    \[\Phi_\alpha(p_j,F_{n_1^+}(p_j))-\Phi_\alpha(q_j,F_{n_1^+}(q_j))\gtrsim_t q_j-p_j.\qedhere\]
\end{proof}

We digress briefly at this point to discuss the geometry behind Lemma \ref{lem:good osc control for curve proj}. Let $\mathbf{a}=(a_1,a_2),\mathbf{b}=(b_1,b_2)\in\mathbb R^2$. We may assume $b_1>a_1$ and $\Phi_\alpha(\mathbf{a})-\Phi_\alpha(\mathbf{b})>0$. Note that $\mathbf{a}+\Gamma$ and $\mathbf{b}+\Gamma$ will either intersect at a point or be disjoint.

First, suppose $(\mathbf{a}+\Gamma)\cap(\mathbf{b}+\Gamma)\ne\varnothing$. That is, there exists $s_0,t_0\in \mathbb R$ such that
\[\mathbf{x}:=(a_1,a_2)+(s_0,\gamma(s_0))=(b_1,b_2)+(t_0,\gamma(t_0)).\]
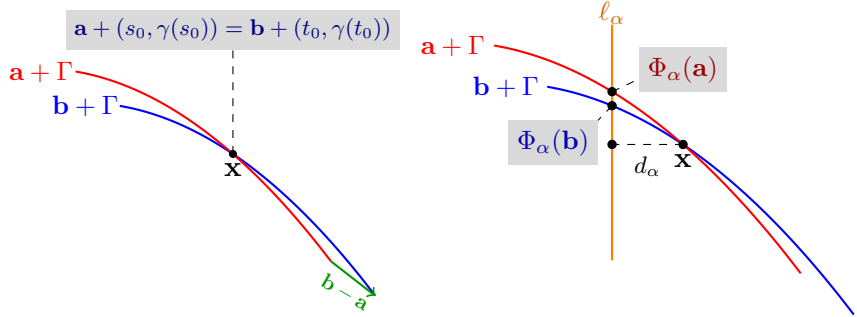
\begin{figure}[H]
    \centering
    \begin{tikzpicture}[scale=1.3]

             \draw[blue, thick, domain=-1.2:1.4, smooth, variable=\t]
             plot ({\t+0.45}, {-0.22*\t*\t-0.7*\t+0.45});
             \draw[red, thick, domain=-1.2:1.4, smooth, variable=\t]
             plot ({\t}, {-0.22*\t*\t-0.7*\t+0.8});

             \filldraw[black] (0.40177,0.48325) circle (1pt) node[below] {\large $\mathbf{x}$};

             \draw[dashed] (0.40177,0.48325) -- (0.40177,1.6);

             \draw[->, thick, green!60!black] (1.4,-0.6112) -- (1.85,-0.9612) node[midway, below, rotate=-30, scale=0.8, text=green!60!black] {$\mathbf{b} - \mathbf{a}$};

             \node[blue] at (-1.1,0.97) {\normalsize $\mathbf{b} + \Gamma$};
             \node[red] at (-1.55,1.32) {\normalsize $\mathbf{a} + \Gamma$};

             \node[fill=gray!30, text=blue!50!black, scale=0.95] at (0.4,1.75)
             {\small $\mathbf{a} + (s_0, \gamma(s_0)) = \mathbf{b} + (t_0, \gamma(t_0))$};


             \begin{scope}[xshift=4.5cm, yshift=0cm, scale=1.2]
             \draw[blue, thick, domain=-1.2:1.4, smooth, variable=\t]
             plot ({\t+0.45}, {-0.22*\t*\t-0.7*\t+0.45});
             \draw[red, thick, domain=-1.2:1.4, smooth, variable=\t]
             plot ({\t}, {-0.22*\t*\t-0.7*\t+0.8});

             \draw[orange, thick] (-0.2,-0.5) -- (-0.2,1.5);
             \node[orange!80!black] at (-0.2,1.6) {\normalsize $\ell_\alpha$};
             \filldraw[black] (0.40177,0.48325) circle (1pt) node[below] {\large $\mathbf{x}$}; 
             \filldraw[black] (-0.2,0.48325) circle (1pt); 
             \filldraw[black] (-0.2,0.9312) circle (1pt); 
             \filldraw[black] (-0.2,0.81205) circle (1pt); 

             \draw[dashed] (0.40177,0.48325) -- (-0.2,0.48325); 
             \node[scale=1] at (0.101,0.3) {\small $d_\alpha$};

             \draw[dashed] (-0.2,0.81205) -- (-0.5,0.5); 
             \draw[dashed] (-0.2,0.9312) -- (0.2,1.12);  

             \node[blue] at (-1.1,0.97) {\normalsize $\mathbf{b} + \Gamma$};
             \node[red] at (-1.55,1.32) {\normalsize $\mathbf{a} + \Gamma$};

            \node[fill=gray!30, text=blue!70!black, scale=1] at (-0.7,0.5)
            {$\Phi_\alpha(\mathbf{b})$};
            \node[fill=gray!30, text=red!60!black, scale=1] at (0.4,1.12)
            {$\Phi_\alpha(\mathbf{a})$};
            \end{scope}
    \end{tikzpicture}
    \caption{$\mathbf{x},\ d_\alpha,\ \Phi_\alpha(\mathbf{a})$, and $\Phi_\alpha(\mathbf{b})$}\label{intersc}
\end{figure}

Comparing coordinates, we have $x_1=s_0+a_1=t_0+b_1$ and $x_2=a_2+\gamma(s_0)=b_2+\gamma(t_0)$.
Set
\[d_\alpha:=\text{dist}(\mathbf{x},\ell_\alpha)=|\alpha-x_1|.\]
Since $b_1-a_1>0$ and $\Phi_\alpha(\mathbf{a})-\Phi_\alpha(\mathbf{b})>0$, by the convexity of $\Gamma$, $d_\alpha=x_1-\alpha$ and $\mathbf{b}$ must lie to the right of $\mathbf{a}$.

We can show that
\[\Phi_\alpha(\mathbf{a})-\Phi_\alpha(\mathbf{b})\sim d_\alpha(b_1-a_1).\]
This estimate suggests that $d_\alpha$ captures the difference between $\Phi_\alpha(\mathbf{a})$ and $\Phi_\alpha(\mathbf{b})$.

Let $J_j=[p_j,q_j]$ be a good interval. Consider the points $\mathbf{a}=(p_j,F_{n_1^+}(p_j))$ and $\mathbf{b}=(q_j,F_{n_1^+}(q_j))$. We do not know the exact value of $|F_{n_1^+}(p_j)-F_{n_1^+}(q_j)|$. However, since $J_j$ is a good interval and the length $q_j-p_j=r_{n_1^+}$ is fixed, this quantity is bounded. Therefore, $d_\alpha$ is controlled by $|F_{n_1^+}(p_j)-F_{n_1^+}(q_j)|$. Geometrically, this can be interpreted as a vertical shift in a fixed range of the curves.

Next, suppose $(\mathbf{a}+\Gamma)\cap(\mathbf{b}+\Gamma)=\varnothing$. In this case we apply a vertical shift to $\mathbf{a}+\Gamma$ to reduce to the intersection case. In fact, since $\Lambda$ is compact, we may define the shift by
\[d:=\min_{\lambda\in \Lambda} \left(\Phi_\lambda(\mathbf{a})-\Phi_\lambda(\mathbf{b})\right).\]
In the case where $\mathbf{a}=(p_j,F_{n_1^+}(p_j))$ and $\mathbf{b}=(q_j,F_{n_1^+}(q_j))$, the quantity $d$ is again controlled by $q_j-p_j$. The reason is the same as in the intersection case.

In fact, from the discussion above, we see that the curve projection behaves similar to the orthogonal projection in the sense that it separates points at a predictable rate. Furthermore, we can deduce that
\[\mathcal{L}^1(\{\lambda\in\Lambda:|\Phi_\lambda(\mathbf{a})-\Phi_\lambda(\mathbf{b})|\leqslant\delta|\mathbf{a}-\mathbf{b}|\})\lesssim\delta,\quad\text{for all}\;\;\mathbf{a},\mathbf{b}\in E,\delta>0,\]
which is the \textit{transversal condition}. Geometrically, it means that for any two distinct points $\mathbf{a},\mathbf{b}\in E$, the set of parameters $\lambda$ for which their projections are nearly aligned has small measure. See \cite[Definition~1.3]{BT20} and detailed proof in \cite[Lemma~2.2]{BT20}.

Returning now to the main proof, we first state the Rising Sun Lemma in this case, which is symmetric to Lemma \ref{rising sun 1}.

\begin{lem}[Rising Sun Type Lemma 2]\label{rising sun 2}
    Fix $K_m$. Then at least one of the following statements is true:
    \begin{enumerate}[\rm(1)]
        \item There are at most $O(N^{1/100})$ intervals $J_j$ that reach $K_m$.
        \item Of all the intervals $J_j$ that reach $K_m$, the proportion of those $J_j$ that are bad is $\gtrsim N^{-1/100}$.
    \end{enumerate}
\end{lem}

\begin{proof}
    By Lemma \ref{lem:good osc control for curve proj}, for every good interval $J_j$,
    \[\mathcal{L}^{1}(\{\Phi_\alpha(x,F_{n_1^+}(x)):x\in J_j\})\sim_t \mathcal{L}^1(J_j).\]
    Since $K_m$ has length $100N^{1/100}r_{n_1^+}$, any consecutive string $J_{j+1},J_{j+2},\cdots,J_{j+p}$ of good intervals that reach $K_m$ can have length at most $p=O(N^{1/100})$. The rest of proof is same as Lemma \ref{rising sun 1}.
\end{proof}

\begin{figure}[H]
\centering
\begin{tikzpicture}[x=0.4cm,y=0.4cm,line cap=round,line join=round]

\draw[->,line width=0.6pt] (0,0) -- (26,0);
\draw[->,line width=0.6pt] (0,0) -- (0,14);

\draw[dashed,dash pattern=on 4pt off 4pt,line width=0.5pt] (0,10) -- (26,10);
\draw[dashed,dash pattern=on 4pt off 4pt,line width=0.5pt] (0,6)  -- (26,6);

\draw[decorate,decoration={brace,amplitude=6pt}, thick] (0,6) -- (0,10);
\node[left] at (-0.5,8) {$K_m$};

\foreach \i in {1,...,25} {
  \draw[line width=0.45pt] (\i,0.3) -- (\i,-0.3);
}

\node at (1.5,0.35) {$g$};
\node at (2.5,0.35) {$g$};
\node at (3.5,0.35) {$g$};
\node at (4.5,0.4) {$b$};
\node at (5.5,0.4) {$b$};

\node at (8.5,0.35) {$g$};
\node at (9.5,0.4) {$b$};

\node at (10.5,0.4) {$b$};
\node at (11.5,0.35) {$g$};
\node at (12.5,0.35) {$g$};
\node at (13.5,0.35) {$g$};

\node at (16.5,0.4) {$b$};
\node at (17.5,0.35) {$g$};
\node at (18.5,0.4) {$b$};
\node at (19.5,0.4) {$b$};
\node at (20.5,0.4) {$b$};
\node at (21.5,0.4) {$b$};

\draw[line width=0.6pt]
  (1,11)
    .. controls (1.33,10.7) and (1.66,10.2) .. (2,9.5)
    .. controls (2.33,9.2) and (2.66,8.7) .. (3,8)
    .. controls (3.33,7.7) and (3.66,7.2) .. (4,6.5)
    .. controls (4.33,7.2) and (4.66,7.7) .. (5,8)
    .. controls (5.33,9) and (5.66,9.7) .. (6,10.2)
    .. controls (6.33,11.2) and (6.66,11.7) .. (7,12)
    .. controls (7.33,11.8) and (7.66,11.3) .. (8,10.2)
    .. controls (8.33,9.9) and (8.66,9.2) .. (9,8)
    .. controls (9.33,8.5) and (9.66,8.8) .. (10,8.9)
    .. controls (10.33,10) and (10.66,10.5) .. (11,10.8)
    .. controls (11.33,10.6) and (11.66,10) .. (12,8.5)
    .. controls (12.33,8.3) and (12.66,8) .. (13,7.6)
    .. controls (13.33,7.5) and (13.64,7.1) .. (14,6)
    .. controls (14.33,5.9) and (14.66,5.7) .. (15,5)
    .. controls (15.33,4.8) and (15.66,4.3) .. (16,3.3)
    .. controls (16.33,6.5) and (16.66,7) .. (17,7.5)
    .. controls (17.33,7.4) and (17.66,7.2) .. (18,6.5)
    .. controls (18.33,7) and (18.66,7.2) .. (19,7.3)
    .. controls (19.66,9) and (20.33,9.5) .. (21,9.8)
    .. controls (21.33,11) and (21.66,11.4) .. (22,11.6)
    .. controls (22.66,12.6) and (23.33,12.9) .. (24,13);

\end{tikzpicture}
\caption{This is a graph of the function $x\mapsto\Phi_\alpha(x,F_{n_1^+}(x))$. We distinguish the intervals $J_j$ that reach an interval $K_m$. The good intervals $J_j$ that reach $K_m$ are marked with “$g$”, while the bad intervals are marked with “$b$”. Note that the function graph is concave down because the function $x\mapsto\gamma'(\alpha-x)$ is increasing.}
\end{figure}
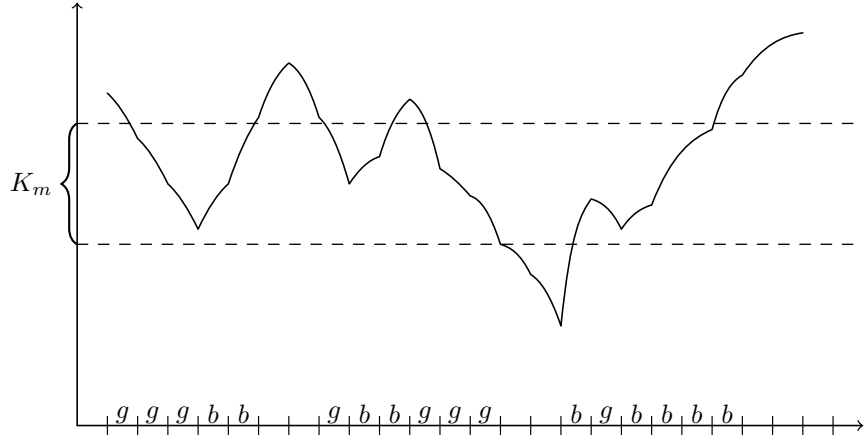

\subsection{Conclusion by contradiction}\label{sec:contradiction}
We proceed by contradiction. Assume that the conclusion in Theorem \ref{thm:quant curve two} does not hold. Then $\mathcal{L}^1(A)\gtrsim N^{-1/100}$, so
\[\mathcal{L}^1\left(A_{n_0^-}\right)\gtrsim N^{-1/100}.\]
We first localize $I_0\subseteq A_{n_0^-}$ with length $r_{n_1^-}$ satisfying \eqref{eqn:Delta A inter I0} and \eqref{eqn:L2 sum on I0}.

\begin{lem}[Well-behaved Dyadic Interval]{}
    There exists a dyadic interval $I_0\subseteq A_{n_0^-}$ with length $r_{n_1^-}$ such that

\begin{equation}\label{eqn:I_0}
        \int_{I_0} \ N^{2/100}\chi_{\Delta A}+N^{3/100}\left(\left|F_{n_1^+}'-F_{n_1^-}'\right|^2+\left|\gamma_{\alpha,n_1^+}'-\gamma_{\alpha,n_1^-}'\right|^2+\left|\gamma_{\alpha',n_1^+}'-\gamma_{\alpha',n_1^-}'\right|^2.\right)\,d\mathcal{L}^1\lesssim \mathcal{L}^1(I_0).
    \end{equation}
    Therefore, \eqref{eqn:Delta A inter I0} and \eqref{eqn:L2 sum on I0} hold on $I_0$.
\end{lem}

\begin{proof}

    We apply quantitative Lebesgue differentiation theorem and quantitative Rademacher's differentiation theorem to find such $I_0$. By \eqref{eqn:Delta A bound} and \eqref{eqn:L2 derivative diff}, we have
    \[\int_{A_{n_0^-}} N^{2/100}\chi_{\Delta A}\,d\mathcal{L}^1\lesssim \mathcal{L}^1(A_{n_0^-}),\]
    and
    \[\int_{A_{n_0^-}} N^{3/100}\left(\left|F_{n_1^+}'-F_{n_1^-}'\right|^2+\left|\gamma_{\alpha,n_1^+}'-\gamma_{\alpha,n_1^-}'\right|^2+\left|\gamma_{\alpha',n_1^+}'-\gamma_{\alpha',n_1^-}'\right|^2.\right)\,d\mathcal{L}^1\lesssim \,\mathcal{L}^1(A_{n_0^-}).\]
    Combining these two bounds gives
    \[\int_{A_{n_0^-}} \ N^{2/100}\chi_{\Delta A}+N^{3/100}\left(\left|F_{n_1^+}'-F_{n_1^-}'\right|^2+\left|\gamma_{\alpha,n_1^+}'-\gamma_{\alpha,n_1^-}'\right|^2+\left|\gamma_{\alpha',n_1^+}'-\gamma_{\alpha',n_1^-}'\right|^2.\right)\,d\mathcal{L}^1\lesssim \,\mathcal{L}^1(A_{n_0^-}).\]
    Note that $A_{n_0^-}$ is the union of dyadic intervals of length $r_{n_0^-}$, so it is also the union of dyadic intervals of length $r_{n_1^-}$. By the pigeonhole principle, there exists a dyadic interval $I_0\subseteq A_{n_0^-}$ of length $r_{n_1^-}$ such that \eqref{eqn:I_0} holds.
\end{proof}

We aim to derive a contradiction by showing the small projection assumption fails for $\alpha$. To achieve this goal, we will apply the Rising Sun Lemma. We only prove the case of $t>0$, because the argument for the case of $t<0$ is essentially the same.

Recall we partition $\mathbb R$ into $K_m$, and use $K_m$ to study the image of the linear approximation of projection. We will argue that a large proportion of $K_m$ is in $\mathcal{N}_{r_{n_1}}(\Phi_\alpha(E))$. We start by classifying $K_m$ by the Rising Sun Lemma.

\begin{dfn}[Low Multiplicity and High Multiplicity Interval]
    An interval $K_m$ is said to be \textit{low multiplicity} if case (1) of Lemma \ref{rising sun 1} holds, and \textit{high multiplicity} otherwise.
\end{dfn}
We claim that there are only few high-multiplicity intervals $K_m$. To make this precise, we relate the multiplicity of the $K_m$ to the behavior of the intervals $J_j$ in the Rising Sun Lemma. For this purpose, we introduce the notions of typical and atypical intervals.

\begin{dfn}[Typical and Atypical Interval]
    An dyadic interval $J_j$ is said to be \textit{typical} if it only reaches low multiplicity intervals $K_m$, and \textit{atypical} if it reaches at least one high multiplicity interval $K_m$.
\end{dfn}

Heuristically, the number of typical intervals is relatively large.

\begin{lem}[Upper Bound for the Number of Atypical Intervals]{}
    At most $O(N^{-2/100}S)$ of the intervals $J_j$ are atypical.
\end{lem}

\begin{proof}
    By Lemma \ref{rising sun 1}, given a high multiplicity interval $K_m$, of all the atypical intervals $J_j$ that reach this $K_m$, the proportion of those $J_j$ that are bad is $\gtrsim N^{-1/100}$. By Lemma \ref{lem:maximal number reach}, each $J_j$ reaches either one or two $K_m$, so the proportion of the atypical intervals that are bad is $\gtrsim N^{-1/100}$. By Lemma \ref{lem:upper bound on bad 1}, at most $O(N^{-3/100}S)$ of the intervals $J_j$ are bad, so at most $O(N^{-2/100}S)$ of the intervals $J_j$ are atypical.
\end{proof}

Since a $J_j$ is typical if it only reaches low multiplicity $K_m$, we can shift our focus to such low multiplicity intervals. However, we need to distinguish between $J_j$'s that intersect $A$ and those that do not. If $A\cap J_j=\varnothing$, $J_j \subseteq \Delta A$. Note that by \eqref{eqn:Delta A inter I0}, at most $O(N^{-2/100}S)$ of intervals $J_j$ do not intersect $A$. Thus, all but $O(N^{-2/100}S)$ of intervals $J_j$ that intersects $A$ are typical.

Furthermore, if $J_j$ is typical and $A\cap J_j\ne\varnothing$, we fix a $x_j\in A\cap J_j$ such that $F_{n_1^+}(x_j)+\gamma_{\alpha,n_1^+}$ lies in a low-multiplicity interval $K_m$. We denote the collection of all such low-multiplicity intervals $K_m$ as $\mathscr{K}$.  Recall we claimed that the size of $\mathscr{K}$ is relatively large. We verify that claim now.

\begin{lem}[Lower Bound for the Size of $\mathscr{K}$]\label{lem:lower bound for Km}
    $|\mathscr{K}|\gtrsim N^{-1/100}S$.
\end{lem}

\begin{proof}
    First, we prove the number of typical intervals $J_j$ intersect $A$ is $\sim S$. It suffices to prove that the number of typical intervals $J_j$ intersect $A$ is $\gtrsim S$. By the 
preceding argument, it is clear that the number of typical interval $J_j$ such that $A\cap J_j\ne\varnothing$ is at least $S(1-CN^{-2/100})$, where $C>0$ is the implicit constant. For sufficiently large $N$, $1-CN^{-2/100}>\frac{1}{2}$, so the number of typical and $A$-nonempty interval $J_j$ is $\gtrsim S$.

    Now we show that $|\mathscr{K}|\gtrsim N^{-1/100}S$. By Lemma \ref{rising sun 1}, there are at most $O(N^{1/100})$ intervals $J_j$ that reach a low-multiplicity interval $K_m$. Then for every low-multiplicity interval $K_m$,
    \[\sum_{K_m\in\mathscr{K}}\hspace{-3mm}\#\left\{x_j\in A\cap J_j:J_j\hspace{.5mm} \text{typical\hspace{-.5mm} and}\ \left(F_{n_1^+}\hspace{-.5mm}+\hspace{-.5mm}\gamma_{\alpha,n_1^+}\right)(x_j)\in K_m\right\}\lesssim N^{1/100}|\mathscr{K}|.\]
    On the other hand, the total count of $x_j$ is same as the number of typical intervals $J_j$ that intersect $A$. Thus, $|\mathscr{K}|\gtrsim N^{-1/100}S$.
\end{proof}

Next, we verify $K_m\in \mathcal{N}_{r_{n_1}}(\Phi_\alpha(E))$ for $K_m\in\mathscr{K}$.

\begin{lem}[Neighborhood Containment]\label{lem:K_m in ngbh 1}
    Let $K_m\in\mathscr{K}$. Then $K_m\subseteq \mathcal{N}_{r_{n_1}}(\Phi_\alpha(E))$.
\end{lem}

\begin{proof}
    Let $y\in K_m$. Then
    \[\left|F_{n_1^+}(x_j)+\gamma_{\alpha,n_1^+}(x_j)-y\right|\leqslant 100N^{1/100}r_{n_1^+}.\]
    By Lemma \ref{lem:L infty difference}, $\left\|F-F_{n_1^+}\right\|_{L^\infty([0,1])}\leqslant N^{1/100}r_{n_1^+}$ and $\left\|\gamma_\alpha-\gamma_{\alpha,n_1^+}\right\|_{L^\infty([0,1])}\leqslant r_{n_1^+}$, so
    \[\left|\Phi_\alpha(x_j,F(x_j))-F_{n_1^+}(x_j)-\gamma_{\alpha,n_1^+}(x_j)\right|\leqslant \left|F(x_j)-F_{n_1^+}(x_j)\right|+\left|\gamma(\alpha-x_j)-\gamma_{\alpha,n_1^+}(x_j)\right|\lesssim N^{1/100}r_{n_1^+}.\]
    Thus, $\left|y-\Phi_\alpha(x_j,F(x_j))\right|\lesssim N^{1/100}r_{n_1^+}$. Since $x_j\in A$, we have $\Phi_\alpha(x_j,F(x_j))\in \mathcal{N}_{r_N}(\Phi_\alpha(E))$, so
    \[d(y,\Phi_\alpha(E))\lesssim N^{1/100}r_{n_1^+}\leqslant 2^{-N^{-9/100}N}N^{1/100}r_{n_1}.\]
    For sufficiently large $N$, $C 2^{-N^{-9/100}N}N^{1/100}<1$, where $C>0$ is the implicit constant. Thus, $y\in \mathcal{N}_{r_{n_1}}(\Phi_\alpha(E))$.
\end{proof}

Now we are ready to reach a contradiction. By Lemma \ref{lem:lower bound for Km} and \ref{lem:K_m in ngbh 1},
\[\mathcal{L}^1(\mathcal{N}_{r_{n_1}}(\Phi_\alpha(E)))\gtrsim N^{-1/100}SN^{1/100}r_{n_1^+}=r_{n_1^-}.\]
This contradicts the assumption $\mathcal{L}^1(\mathcal{N}_{r_{n+1}}(\Phi_\alpha(E)))\leqslant r_n$ for sufficiently large $N$.

\begin{rmk}\label{rmk:case YX}
    To prove
    \[\mathcal{L}^1(\{x\in[0,1]:\exists\,y\in[-r_N,r_N]\ \text{such that}\ (F(x)+y,x)\in E\})\lesssim N^{-1/100},\]
    the only modification occurs in the quantitative version of Rademacher's differentiation theorem. We need to quantify the maps $x\mapsto \Phi_{\lambda}(F(x),x)$ for $\lambda\in\{\alpha,\alpha'\}$. In the counting argument, set
    \[t=\nabla \Phi_\alpha(F(a),a)\cdot (c,1)=(-\gamma'(\alpha-F(a)),1)\cdot (c,1).\]
    We then modify the proof of Lemma \ref{lem:upper bound on bad 1} to obtain an upper bound on the number of bad intervals. Note that $\partial_1\Phi_\alpha(x)=-\gamma'(\alpha-x_1)$ is $1$-Lipschitz. This property allows us to adapt the proof of Lemma \ref{lem:upper bound on bad 1}.
\end{rmk}

\bibliography{refs}
\bibliographystyle{abbrv}

\end{document}